\documentclass[12pt]{amsart}

\usepackage[T1]{fontenc}
\usepackage[utf8]{inputenc}
\usepackage{lmodern}
\usepackage{float}
\usepackage{amsmath,amssymb,amsthm,mathtools,booktabs,array}
\mathtoolsset{showonlyrefs=true}

\usepackage[a4paper,left=2cm,right=2cm,top=2cm,bottom=2cm]{geometry}
\usepackage{cite}
\usepackage[unicode=true,pdfencoding=auto,colorlinks=true,linkcolor=blue,citecolor=blue,urlcolor=blue]{hyperref}
\usepackage{microtype,mathrsfs}
\allowdisplaybreaks[2]
\newtheorem{theorem}{Theorem}[section]
\newtheorem{proposition}[theorem]{Proposition}
\newtheorem{lemma}[theorem]{Lemma}
\newtheorem{corollary}[theorem]{Corollary}
\theoremstyle{definition}

\newcommand{\F}{\mathbb F_2}
\newcommand{\A}{\mathcal A}
\newcommand{\GL}{\operatorname{GL}}
\newcommand{\Sq}{\operatorname{Sq}}
\newcommand{\Ext}{\operatorname{Ext}}
\newcommand{\Tr}{\operatorname{Tr}}
\newcommand{\im}{\operatorname{im}}
\newcommand{\rank}{\operatorname{rank}}
\newcommand{\Span}{\operatorname{Span}_{\F}}
\newcommand{\DP}{\operatorname{DP}}
\newcommand{\coker}{\operatorname{coker}}
\newcommand{\II}{\mathrm{II}}
\newcommand{\III}{\mathrm{III}}

\title[Computer-assisted proofs for the fourth algebraic transfer]
{Computer-assisted proofs for the fourth algebraic transfer in certain generic degrees}
\author{\DJ\d{\u a}ng V\~o Ph\'uc}
\address{Department of Mathematics, FPT University, An Phu Thinh New Urban Area, Quy Nhon, Vietnam}
\email{dangphuc150488@gmail.com}
\thanks{ORCID: \url{https://orcid.org/0000-0002-6885-3996}.}
\subjclass[2020]{55S10,55S05,55T15,68W30}
\keywords{Steenrod algebra, hit problem, algebraic transfer, modular representations}

\begin{document}

\begin{abstract}
This paper restructures the author's earlier article in \emph{Journal of Homotopy and Related Structures}~\cite{PhucJHRS2024} to correct its erroneous invariant-dimension calculations for the fourth Singer algebraic transfer. More precisely, we determine the domain and image of the transfer in the degree families $2^{s+t+1}+2^{s+1}-3$ with $t\ne3$ and $2^{s+t}+2^s-2$ with $t\notin\{2,3,4\}$, where $s,t$ are positive integers. Divided-power spikes are basis elements dual to polynomial monomials whose exponents have the form $2^j-1$ for nonnegative integers $j$. Their translates under $\GL_4(\mathbb F_2)$, the group of invertible $4\times4$ matrices over the field with two elements, generate submodules of the Steenrod kernel, the subspace annihilated by all positive Steenrod squares. We present these submodules as quotients of sums of permutation modules on complete flags, which are chains in $\mathbb F_2^4$ containing one subspace of each dimension $0,1,2,3,4$. Identifying the relation spaces in these presentations gives explicit equivariant comparison isomorphisms. A recursion obtained by separating even and odd exponents in each variable reduces the comparisons to specified identities of finite-dimensional subspaces. These comparisons apply to both complete degree families, including the parameter values omitted from the transfer theorem. Duplication and partial flag embeddings also give theoretical comparisons in specified parameter ranges. Comparison with Sum's dimensions of the polynomial algebra modulo the sum of the images of positive Steenrod squares \cite[Theorem~1.1]{Sum2014} describes precisely when spikes generate the full Steenrod kernel. Together with five direct invariant calculations and the fourth cohomology of the Steenrod algebra, this proves injectivity in the stated ranges. As a consequence, Singer's injectivity conjecture holds in rank four in both degree families under the stated parameter restrictions. The cokernel is one-dimensional in homological degree $256$ and zero in every other degree in those ranges.
\end{abstract}

\maketitle

\section{Introduction}\label{sec:intro}
Work over the field $\F$ with two elements. All vector spaces, linear
duals and tensor products are over $\F$, except when another coefficient
ring is displayed. For a positive integer $k$, let
$P(k)=\F[x_1,\ldots,x_k]$, with $\deg x_i=1$, be the polynomial
algebra with its usual action of the mod-$2$ Steenrod
algebra $\A$. Peterson's hit problem asks for a minimal set of
$\A$-module generators of $P(k)$~\cite{Peterson1987}.
Equivalently, one seeks a basis of the quotient
$Q(k)=P(k)/\A^+P(k)$, where $\A^+$ is the ideal of positive-degree
elements of $\A$. An element of $\A^+P(k)$ is called \emph{hit},
and $Q_n(k)$ is the \emph{cohit space} in degree $n$.
For $n\geq0$, the subscripts in $P_n(k)$ and $Q_n(k)$ denote
the homogeneous components of degree $n$; put $P_n(k)=0$ for $n<0$.
The four-variable problem
has been solved by Sum
\cite[Theorem~1.1 and Sections~3--8]{Sum2014}; see also
\cite[Theorem~1.4]{Sum2015}. The general problem remains open in higher ranks.

Write $\GL_k=\GL_k(\F)$ for the group of invertible $k\times k$
matrices over $\F$. It acts on $P(k)$ by linear substitutions.
This action commutes with the Steenrod action and therefore induces an
action on $Q_n(k)$. For a vector space $M$, write $M^*$ for its
linear dual. The induced action on a dual module is the contragredient
action, defined by $(g\lambda)(v)=\lambda(g^{-1}v)$ for
$\lambda\in M^*$ and $v\in M$.
The graded dual $\DP(k)=\bigoplus_{n\geq0}P_n(k)^*$ is the
divided-power algebra. Its basis elements dual to the polynomial
monomials are called divided-power monomials; those dual to
degree-$n$ monomials have homological degree $n$.
The quotient defining $Q_n(k)$ identifies $K_n(k)=Q_n(k)^*$
with a $\GL_k$-submodule of $\DP(k)$.
It consists of the degree-$n$ elements annihilated by all positive
Steenrod squares.
Write $\Span$ for linear span over $\F$.
For a group $\Gamma$ acting on a vector space $M$, its invariant
space and coinvariant quotient are
\[
 M^\Gamma=\{v\in M:gv=v\text{ for every }g\in\Gamma\},\qquad
 M_\Gamma=M/\Span\{gv-v:g\in\Gamma,\ v\in M\},
\]
respectively. Write $\F[\Gamma]$ for the group algebra of $\Gamma$.
The group $\Ext_{\A}^{p,q}(\F,\F)$ is the cohomology of $\A$
in cohomological degree $p$ and internal degree $q$, with trivial
coefficient module $\F$. A star in either index denotes the direct sum
over that index. Products of cohomology classes are taken in the
bigraded algebra $\Ext_{\A}^{*,*}(\F,\F)$.
Singer's cohomological algebraic transfer is the map
\begin{equation}\label{eq:singer-transfer}
 \Tr_{k,n}:(K_n(k))_{\GL_k}
 =\F\otimes_{\F[\GL_k]}K_n(k)
 \longrightarrow\Ext_{\A}^{k,n+k}(\F,\F),
\end{equation}
where $\F$ is the trivial right $\F[\GL_k]$-module in the tensor product.
The dual of the transfer domain is $Q_n(k)^{\GL_k}$.
The cokernel of a transfer map is its target modulo its image.
Singer proved that the transfer is an isomorphism in ranks
one and two~\cite{Singer1989}, and Boardman established the rank-three
case~\cite{Boardman1993}. Singer's injectivity conjecture asks whether
\eqref{eq:singer-transfer} is injective in every rank and degree.
In rank four, Lin's determination of $\Ext_{\A}^{4,*}(\F,\F)$
provides the target~\cite{Lin2008}; the domain requires information
about the $\GL_4$-module structure of the cohit spaces.

A \emph{divided-power spike} is the basis element dual to a polynomial
monomial whose exponents are all of the form $2^j-1$ for
nonnegative integers $j$.
Write $J_n(k)$ for the submodule of $K_n(k)$ generated by the
$\GL_k$-translates of these monomials in degree $n$.
The divided-power notation and the inclusion in the Steenrod kernel
are recalled in Section~\ref{sec:presentations}.

Let $\mu(0)=0$. For $n>0$, let $\mu(n)$ be the least positive
integer $r$ such that $n=\sum_{i=1}^{r}(2^{a_i}-1)$ for positive
integers $a_1,\ldots,a_r$~\cite[Definition~2.4.2]{WW1}.
Wood's vanishing theorem gives $Q_n(4)=0$ whenever $\mu(n)>4$
\cite{Wood1989}. For $m\geq0$, Kameko's squaring map
$\kappa:Q_{2m+4}(4)\longrightarrow Q_m(4)$ is induced by the
linear map that sends a monomial $x_1x_2x_3x_4y^2$ to $y$
for a monomial $y$ of degree $m$, and every other monomial to zero.
It is $\GL_4$-equivariant and
is an isomorphism whenever $\mu(2m+4)=4$~\cite{Kameko1990};
see also~\cite[Propositions~6.5.2--6.5.3]{WW1}.
Since these isomorphisms strictly decrease the degree, the calculation
of nonzero cohit modules, their invariant spaces and the corresponding
transfer domains reduces to degrees with $\mu(n)<4$
\cite[Proposition~6.5.6]{WW1}. Apart from degree zero, these degrees have four forms: $2^{s+1}-r$
with $r\in\{1,2,3\}$, the two degree families 
\begin{equation}\label{eq:families}
 n_{\II}(s,t)=2^{s+t+1}+2^{s+1}-3,
 \qquad
 n_{\III}(s,t)=2^{s+t}+2^s-2,
 \qquad s,t\geq1,
\end{equation}
and the three-parameter family $2^{s+t+u}+2^{s+t}+2^s-3$,
where $s,t,u$ are positive integers~\cite[p.~2]{Sum2014}.
The grouping and the subscripts in~\eqref{eq:families} agree with
\cite[Section~2, equation~(2.1)]{Phuc26a}.
The two families in~\eqref{eq:families} have $\mu$-values three
and two, respectively. The present paper determines the fourth
transfer for $n_{\II}(s,t)$ with $t\ne3$ and for $n_{\III}(s,t)$
with $t\notin\{2,3,4\}$.

The fourth transfer in degrees $2^{s+1}-r$, $r\in\{1,2,3\}$,
is determined in Sum's corrected paper~\cite[Theorem~4.1]{Sum2018}
and discussed in
\cite[Section~2, equations~(2.2)--(2.3)]{Phuc26a}.
Corrections to the invariant calculations in~\cite{Phuc26a} are
recorded in~\cite{Phuc26b,Phuc26c}.
In particular, the addendum~\cite[Section~4]{Phuc26c} considers
$n_{\II}(s,3)$ and the $\GL_4$-invariants in the kernel of
Kameko's squaring map for $n_{\III}(s,3)$.

The three-parameter family $2^{s+t+u}+2^{s+t}+2^s-3$, together
with the columns $n_{\III}(s,2)=5\cdot2^s-2$ and
$n_{\III}(s,4)=17\cdot2^s-2$, was considered in the rank-four
part of the author's article~\cite{PhucPalermo2025}.
However, following publication, we discovered errors in the manual calculations of \cite{PhucPalermo2025}, prompting the withdrawal of that paper, as noted in \cite{PhucRetraction2026}.
Its treatment of these families has been corrected and expanded
with additional results in the 112-page
manuscript~\cite{Phuc26d}.
A separate treatment of these same three families is given
in~\cite[Theorems~1.1--1.2]{Phuc26e}.

For the cohit spaces in~\eqref{eq:families},
Sum's dimension formulas cover every positive parameter pair
\cite[Theorem~4]{Sum2013}; complete proofs appear in
\cite[Theorem~1.1 and Sections~6--7]{Sum2014} and
\cite[Theorem~1.4]{Sum2015}.
His later work determines the invariant dimensions and proves injectivity
of the fourth transfer in both complete families
\cite[Theorems~1.2 and~1.4, Corollary~1.6]{Sum2025}.
His first family is written $d_{u,t}=2^{u+t}+2^u-3$, so that
$n_{\II}(s,t)=d_{s+1,t}$; his second family uses the parameters in
\eqref{eq:families}
\cite[Theorems~1.2 and~1.4]{Sum2025}.

The aim of the present paper is to correct, revise, and restructure the author's earlier article~\cite{PhucJHRS2024}. The revised argument corrects the
invariant-dimension calculations and supplies the proofs for the
degree families considered here.
The invariant-dimension calculations in~\cite{Phuc26a,PhucJHRS2024}
contain errors identified by Sum
\cite[Theorem~1.2 and Remark~1.3]{Sum2025}.
For example, the correct values are
\[
 \dim Q_{21}(4)^{\GL_4}=0,
 \qquad
 \dim Q_{n_{\II}(s,2)}(4)^{\GL_4}=1\quad(s\geq3).
\]
Sum also corrects the values in the second family at $t=2,4$
\cite[Remarks~1.5 and~4.3.2]{Sum2025}.

We give a corrected computer-assisted proof of the transfer
calculations through explicit presentations of the submodules
generated by divided-power spikes. The transfer dimensions and
injectivity in these families are those determined by Sum
\cite[Theorems~1.2 and~1.4, Corollary~1.6]{Sum2025}.
The main structural statement is an isomorphism of $\GL_4$-modules
with prescribed values on translated spike generators
(Proposition~\ref{prop:stability} in Section~\ref{sec:presentations}).
The construction applies to all parameters in~\eqref{eq:families}.
It also determines the codimension of the spike submodule in the full
Steenrod kernel throughout both families
(Proposition~\ref{prop:KJ} in Section~\ref{sec:families}).
Dual duplication and partial flag embeddings explain several
of these comparisons by established constructions
(Lemma~\ref{lem:global-duplication} and
Propositions~\ref{prop:partial-II}--\ref{prop:partial-III}).
The recursion provides one proof for both complete families,
including the boundary parameter values and the relations between
the cyclic spike images. The evaluation method, its $35$-dimensional
auxiliary space, and the use of simultaneous images were already
developed in~\cite[Section~3.3, Lemma~3.3 and Proposition~3.4]{Phuc26e}.
Here they are applied to the two families in~\eqref{eq:families},
with a complete specification and verification of the finite
identities needed for their parameter reductions.
The resulting transfer calculation is stated below and proved in
Section~\ref{sec:transfer}.

We use the classes
$h_i\in\Ext_{\A}^{1,2^i}(\F,\F)$ and
$c_i\in\Ext_{\A}^{3,11\cdot2^i}(\F,\F)$ for $i\geq0$,
together with $e_0\in\Ext_{\A}^{4,21}(\F,\F)$ and
$D_3(2)\in\Ext_{\A}^{4,260}(\F,\F)$, with the names and
normalizations of Lin~\cite{Lin2008}.

\begin{theorem}\label{thm:main}
Let $s,t$ be positive integers. Suppose that $n=n_{\II}(s,t)$ with $t\ne3$, or
that $n=n_{\III}(s,t)$ with $t\notin\{2,3,4\}$.
Then $\Tr_{4,n}$ is injective. Its image, the dimension
$d=\dim Q_n(4)^{\GL_4}$, and the dimension
$c=\dim\coker\Tr_{4,n}$ are given in
Tables~\ref{tab:II} and~\ref{tab:III}.
It is an isomorphism in every displayed degree other than $256$.
At $n=256$, the cokernel is one-dimensional and is represented by
$D_3(2)$.
\end{theorem}

In the tables, $\varnothing$ denotes the empty basis of the zero image.

\begin{table}[H]
\centering
\caption{The transfer in degrees $n_{\II}(s,t)$.}\label{tab:II}
\begin{tabular}{ccclcc}
\toprule
$t$&$s$&$n$&Basis of the image&$d$&$c$\\\midrule
1&1&9&$h_1c_0$&1&0\\
1&2&21&$\varnothing$&0&0\\
1&$\geq3$&$3\cdot2^{s+1}-3$&$h_0h_{s+1}^3$&1&0\\
2&1&17&$e_0$&1&0\\
2&2&37&$\varnothing$&0&0\\
2&$\geq3$&$5\cdot2^{s+1}-3$&$h_0h_s^2h_{s+3}$&1&0\\
$\geq4$&$1,2$&$n_{\II}(s,t)$&$h_0h_{s+1}h_{s+t}^2$&1&0\\
$\geq4$&$\geq3$&$n_{\II}(s,t)$&$h_0h_s^2h_{s+t+1},\ h_0h_{s+1}h_{s+t}^2$&2&0\\
\bottomrule
\end{tabular}
\end{table}
\begin{table}[H]
\centering
\caption{The transfer in degrees $n_{\III}(s,t)$; $\delta_{t,7}=1$ for $t=7$ and is zero otherwise.}\label{tab:III}
\begin{tabular}{ccclcc}
\toprule
$t$&$s$&$n$&Basis of the image&$d$&$c$\\\midrule
1&$1,2,4$&$4,10,46$&$\varnothing$&0&0\\
1&3&22&$h_2c_1$&1&0\\
1&$\geq5$&$3\cdot2^s-2$&$h_1h_s^3$&1&0\\
$\geq5$&1&$2^{t+1}$&$h_1^2h_t^2$&1&$\delta_{t,7}$\\
$\geq5$&2&$2^{t+2}+2$&$h_1^3h_{t+2}$&1&0\\
$\geq5$&$3,4$&$n_{\III}(s,t)$&$h_0^2h_sh_{s+t},\ h_1h_sh_{s+t-1}^2$&2&0\\
$\geq5$&$\geq5$&$n_{\III}(s,t)$&$h_0^2h_sh_{s+t},\ h_1h_{s-1}^2h_{s+t},\ h_1h_sh_{s+t-1}^2$&3&0\\
\bottomrule
\end{tabular}
\end{table}

The injectivity assertion gives the following consequence for Singer's conjecture.
\begin{corollary}\label{cor:singer}
Singer's injectivity conjecture holds in rank four in every degree
$2^{s+t+1}+2^{s+1}-3$ with $s,t\geq1$ and $t\ne3$, and
in every degree $2^{s+t}+2^s-2$ with $s,t\geq1$ and
$t\notin\{2,3,4\}$.
\end{corollary}

\subsection*{Computer calculations and theoretical arguments}
The comparison of spike modules is made in a fixed permutation module,
that is, a vector space with a basis permuted by the group action.
A complete flag in $\F^4$ is a nested sequence of subspaces containing
one subspace of each dimension $0,1,2,3,4$.
The complete flag module has these flags as a basis, with the action
induced by linear substitutions on $\F^4$.
For one spike, its translates define a surjection from the complete
flag module. For all spikes of a given degree, these surjections combine
into a presentation of $J_n(4)$: a surjection onto this module from
a direct sum of complete flag modules. The kernel records the relations
among the translated generators, including relations between submodules
generated by different spikes. Denote the image of the dual map by $R_n$.
The presentation kernel consists of the vectors on which every
functional in $R_n$ vanishes; this subspace is called the annihilator
of $R_n$. Thus equality of
the dual images at two parameter pairs gives an equivariant
isomorphism with prescribed values on translated spikes
(Proposition~\ref{prop:comparison-criterion}).
The construction of these presentations and the proof of the
comparison criterion are theoretical arguments.

The computer calculations supply three sets of data: the thirteen
subspace equalities in Lemma~\ref{lem:subspace-identities} and
Table~\ref{tab:checks}, the thirty-eight representative spike
dimensions in Tables~\ref{tab:dimensions-II} and~\ref{tab:dimensions-III},
and the five quotient and invariant calculations in degrees
$9,17,21,22,46$ recorded in Lemma~\ref{lem:initial-invariants}.
The first two sets are calculated by separating the even and odd
exponents in each variable and using the dual flag maps; these
constructions are given in Sections~\ref{sec:presentations}
and~\ref{sec:frobenius}. For each subspace equality, exact row
reduction verifies that the two subspaces and their sum have the
same dimension in the same coordinates.
The five remaining calculations construct the hit subspace from
the Cartan formula, form the quotient, and solve the invariance
conditions for generators of $\GL_4$.
All matrices are constructed from these formulas and row-reduced
over $\F$ to determine their ranks. The computations are independent
of the cohit dimensions in Sum's theorem
\cite[Theorem~1.1]{Sum2014}.
The consolidated archive~\cite{Data} contains the subspace
certificates, all thirty-eight representative spike ranks, and the
quotient and invariant data. 

The passage to arbitrary parameters is proved algebraically.
Lemma~\ref{lem:coefficient-recursion} derives the recursion for $R_n$
from the Frobenius decomposition of the polynomial algebra.
This decomposition groups monomials by the parity of their exponents
and writes each group as a monomial with exponents zero or one
multiplied by a square.
Lemma~\ref{lem:family-recursions} gives the sequence of operations
for each complete degree family. The thirteen verified subspace
equalities imply equalities of subsequent iterates by induction.
Writing $n_0$ for the representative degree specified in
Proposition~\ref{prop:stability}, these give $R_n=R_{n_0}$
in a common dual flag module. The comparison criterion supplies the isomorphisms of
Proposition~\ref{prop:stability}, with their equivariance and
specified values on generators.
This deduction uses the spike presentations, the parity recursion
and the verified subspace identities. It therefore establishes the
parameter reductions independently of cohit-dimension formulas.
The thirty-eight representative ranks then determine the spike
dimensions throughout both infinite families.
In addition, Section~\ref{sec:partial-flags} gives proofs using
partial flag embeddings for $s\geq3,t\geq4$ in the first family
and $s\geq5,t\geq4$ in the second. Those proofs derive both the
dimensions and the comparison maps from the embedding theorem,
independently of the finite calculations and of cohit dimensions.

Comparison with the full Steenrod kernel uses a separate theoretical
input: Sum's cohit-dimension theorem
\cite[Theorem~4]{Sum2013}, with complete proofs in
\cite[Theorem~1.1]{Sum2014} and \cite[Theorem~1.4]{Sum2015}.
Proposition~\ref{prop:KJ} compares these values of
$\dim K_n(4)=\dim Q_n(4)$ with the independently calculated
values of $\dim J_n(4)$.
The inclusion $J_n(4)\subseteq K_n(4)$ then determines
$\dim(K_n(4)/J_n(4))$ and proves equality of the two modules
whenever that dimension is zero.
Thus the cohit theorem is used at the comparison step, after the
spike calculations and the proof of the equivariant isomorphisms.
Section~\ref{sec:transfer} proves the relations among spike
coinvariants and identifies their transfer images in
Proposition~\ref{prop:adams-injection}, using Singer's rank-one
calculation and multiplicativity~\cite{Singer1989}.
These arguments, the five directly computed invariant dimensions,
and the cohomology and transfer results of
\cite{Lin2008,Boardman1993,Ha2007} complete the proof of
Theorem~\ref{thm:main}.

\medskip

The paper is structured as follows.
Section~\ref{sec:presentations} introduces the flag presentations and
states the comparison isomorphisms. Section~\ref{sec:families}
describes both complete parameter families and explains the additional
cases outside Theorem~\ref{thm:main}.
Section~\ref{sec:frobenius} derives the recursion, and
Section~\ref{sec:comparisons} proves the comparison isomorphisms
and relates them to duplication and partial flag embeddings.
Finally, Section~\ref{sec:transfer} gives the invariant calculations and
completes the proof of the main result, Theorem~\ref{thm:main}.

\section{Spike submodules and their permutation presentations}
\label{sec:presentations}\label{sec:prelim}
The comparison maps are defined on translated spikes. We first describe
the submodule generated by one spike and its presentation by the
complete flag module. Taking the sum of these presentations makes the
relations between different spike generators explicit and provides
the criterion used to compare degrees.

\subsection{The Steenrod kernel and divided-power spikes}
Write $H_n(k)=(\A^+P(k))_n$ for the hit subspace of $P_n(k)$.
The set $\mathbb Z_{\geq0}^k$ consists of ordered $k$-tuples
of nonnegative integers. For $e=(e_1,\ldots,e_k)$ in this set,
put $x^e=\prod_{i=1}^k x_i^{e_i}$ and $|e|=\sum_{i=1}^k e_i$.
For a nonnegative integer $r$, write $\Sq^r$ for the Steenrod
square of degree $r$. Its action is given by the Cartan formula
\begin{equation}\label{eq:cartan}
 \Sq^r(x^e)=
 \sum_{\substack{r_1,\ldots,r_k\geq0\\r_1+\cdots+r_k=r}}
 \left(\prod_{i=1}^k\binom{e_i}{r_i}\right)
 x_1^{e_1+r_1}\cdots x_k^{e_k+r_k},
\end{equation}
where the summation indices are integers and the ordinary binomial
coefficients are reduced modulo two, with $\binom{a}{b}=0$ for
nonnegative integers $a,b$ satisfying $b>a$.

Set $\DP_n(k)=P_n(k)^*$ and
$\DP(k)=\bigoplus_{n\geq0}\DP_n(k)$.
The divided-power basis consists of
$a^{(e)}=a_1^{(e_1)}\cdots a_k^{(e_k)}$ of degree $|e|$,
with $a_i^{(0)}=1$ and $a_i=a_i^{(1)}$.
The evaluation pairing between the dual space and the polynomial
space is given by $\langle a^{(e)},x^f\rangle=\delta_{e,f}$,
where $\delta_{e,f}=1$ if $e=f$ and is zero otherwise.
Multiplication satisfies
$a_i^{(r)}a_i^{(s)}=\binom{r+s}{r}a_i^{(r+s)}$ for $r,s\geq0$.
For a degree-one element $v=\sum_{i=1}^k c_i a_i$, with $c_i\in\F$,
its $r$th divided power is
\[
 v^{(r)}=\sum_{|e|=r}\left(\prod_{i=1}^k c_i^{e_i}\right)a^{(e)}
 \qquad(r\geq0).
\]
For $z\in\DP_n(k)$ and $0\leq r\leq n$, define
$z\Sq^r\in\DP_{n-r}(k)$ by
$\langle z\Sq^r,f\rangle=\langle z,\Sq^r f\rangle$
for every $f\in P_{n-r}(k)$; set $z\Sq^r=0$ for $r>n$.
This is the right Steenrod action on $\DP(k)$.
Then
\[
 K_n(k)=\{z\in\DP_n(k):z\Sq^r=0\text{ for every }r>0\}
       =Q_n(k)^*.
\]
For $g=(g_{ij})\in\GL_k$, the polynomial action is
$gx_j=\sum_i g_{ij}x_i$.
The action on the graded dual is defined by
$\langle gz,f\rangle=\langle z,g^{-1}f\rangle$;
in degree one it is $ga_i=\sum_j(g^{-1})_{ij}a_j$.
The resulting pairing is invariant and gives
\begin{equation}\label{eq:duality}
 ((K_n(k))_{\GL_k})^*\cong Q_n(k)^{\GL_k}.
\end{equation}
Indeed, a linear functional on $K_n(k)$ descends to coinvariants
exactly when it vanishes on every $gz-z$, which is its invariance
condition. These identifications follow the duality in
Walker--Wood~\cite[Section~9.4]{WW1}, expressed here with left
group actions.

Put $\nu_j=2^j-1$ for $j\geq0$.
A \emph{spike} is a polynomial monomial all of whose exponents
have this form. The dual monomial
$\prod_i a_i^{(\nu_{h_i})}$, with nonnegative integers $h_i$,
is a \emph{divided-power spike},
also called a $d$-spike in~\cite[Definition~9.4.3]{WW1}.
Every such monomial belongs to the Steenrod kernel
\cite[Proposition~9.4.4]{WW1}.
Let $J_n(k)$ be the $\GL_k$-submodule of $K_n(k)$ generated
by all degree-$n$ divided-power spikes
\cite[Definition~9.4.5]{WW1}.
The inclusion $J_n(k)\subseteq K_n(k)$ dualizes to the surjection
\begin{equation}\label{eq:restriction}
 Q_n(k)\longrightarrow J_n(k)^*.
\end{equation}
Thus $J_n(k)=K_n(k)$ precisely when the restriction map
\eqref{eq:restriction} is an isomorphism.

We also use the \emph{dual up Kameko map}, denoted by $\kappa^*$
in~\cite[Definition~9.5.1]{WW1}. Its degree-$n$ component is
\begin{equation}\label{eq:dual-up-kameko}
 \Phi_{k,n}:\DP_n(k)\longrightarrow\DP_{2n+k}(k),\qquad
 a_1^{(e_1)}\cdots a_k^{(e_k)}
 \longmapsto a_1^{(2e_1+1)}\cdots a_k^{(2e_k+1)}.
\end{equation}
The new exponents sum to $2n+k$. Distinct basis elements have
distinct images, so this map is injective. It is $\GL_k$-equivariant
and maps $K_n(k)$ to $K_{2n+k}(k)$ and $J_n(k)$ to $J_{2n+k}(k)$
\cite[Proposition~9.5.4]{WW1}.
On a spike, it sends each exponent $2^j-1$ to $2^{j+1}-1$, including $j=0$.

\subsection{One spike and the complete flag module}\label{sec:spikes}
We now fix $k=4$ and write $G=\GL_4$.
A \emph{height vector} is a weakly decreasing quadruple
$\mathbf h=(h_1,h_2,h_3,h_4)$ of nonnegative integers.
Thus $h_1\geq h_2\geq h_3\geq h_4\geq0$.
It specifies the spike and its degree by
\[
 \sigma_{\mathbf h}=\prod_{i=1}^4a_i^{(2^{h_i}-1)},
 \qquad n(\mathbf h)=\sum_{i=1}^4(2^{h_i}-1).
\]
Its \emph{weight sequence} is
$\omega(\mathbf h)=(\omega_0,\omega_1,\ldots)$, where
\[
 \omega_r=\#\{i:h_i>r\},\qquad
 n(\mathbf h)=\sum_{r\geq0}2^r\omega_r.
\]
Here $\#$ denotes the number of elements in the finite set.
The indices start at zero to agree with binary positions.
Thus the $r$th binary digits of the four spike exponents form
the vector with $\omega_r$ initial ones and subsequent zeros.

Let $B\subseteq G$ be the lower triangular subgroup.
The spike $\sigma_{\mathbf h}$ is fixed by $B$
\cite[Proposition~9.4.7]{WW1}.
The set $B\backslash G=\{Bg:g\in G\}$ consists of the right
cosets of $B$. Let $F=\F[B\backslash G]$ be the vector space
with basis $\{e_{Bg}:Bg\in B\backslash G\}$ and action
\[
 \gamma e_{Bg}=e_{Bg\gamma^{-1}}\qquad(\gamma\in G).
\]
The coset $Bg$ determines a complete flag of row-vector subspaces:
its $i$th subspace is spanned by the first $i$ rows of $g$.
The module $F$ is isomorphic to the complete flag module $\F[G/B]$
under $Bg\mapsto g^{-1}B$, where $\F[G/B]$ has the cosets $gB$
as a basis with the usual left multiplication action.
Its dimension is the index $|G:B|$, the number of cosets:
\[
 |G:B|=(2^1-1)(2^2-1)(2^3-1)(2^4-1)=315.
\]
The cyclic submodule generated by $\sigma_{\mathbf h}$ is the
span of all its $G$-translates. Its presentation is
\begin{equation}\label{eq:single-flag-map}
 J(\mathbf h)=\F[G]\sigma_{\mathbf h},\qquad
 \varphi_{\mathbf h}:F\twoheadrightarrow J(\mathbf h),
 \quad e_{Bg}\longmapsto g^{-1}\sigma_{\mathbf h}.
\end{equation}
The formula is independent of the coset representative because
$B$ fixes the spike, and it is equivariant by the specified action
on $F$. The $315$ cosets index translates with possible repetitions;
an individual spike orbit can be smaller.

The map~\eqref{eq:single-flag-map} is the Crabb--Hubbuck flag map
of~\cite[Definition~23.1.2]{WW2}, expressed with our action conventions.
Flag maps and duplication are studied in~\cite{WW2009};
Powell~\cite[Theorem~5.17 and Proposition~6.1]{Powell2009}
gives embedding criteria and shows that, for strictly decreasing
positive heights, injectivity is preserved when each positive
height is increased by one. These results concern the map attached
to a specified spike. Our simultaneous presentation below records
the relations between the images attached to all spikes in a degree.
The weight of a divided-power monomial is the sequence whose $r$th
entry counts the exponents having a $1$ in binary position $r$.
We compare such sequences lexicographically, starting at position zero:
the first unequal entries determine their order.
Every monomial occurring in $J(\mathbf h)$ has weight at least
$\omega(\mathbf h)$~\cite[Proposition~23.1.4]{WW2}.
The local spike module of Walker--Wood is the image of $J(\mathbf h)$
after quotienting $\DP_{n(\mathbf h)}(4)$ by the span of monomials
whose weights are strictly greater than $\omega(\mathbf h)$
\cite[Definition~9.4.9]{WW1}.
For these local modules, deleting one of two consecutive equal entries
of a weakly decreasing weight sequence induces an equivariant surjection
\cite[Proposition~24.1.1]{WW2}.
Section~\ref{sec:duplication} also uses the global dual duplication
map on Steenrod kernels. Section~\ref{sec:partial-flags} explains
how the local partial flag embedding theorem identifies global
cyclic spike submodules and proves direct sum decompositions.

\subsection{Several spikes and a comparison criterion}\label{sec:coefficients}
Choose an ordering $\mathcal H_n=(\mathbf h_1,\ldots,\mathbf h_m)$
of all height vectors of degree $n$, and write
$\mathbf h_j=(h_{j1},h_{j2},h_{j3},h_{j4})$.
In $F_m=\bigoplus_{j=1}^m F$, let $e_{j,Bg}$ be $e_{Bg}$
in the $j$th summand and zero in the other summands.
Every spike of degree $n$ is a coordinate permutation of one of
the $\sigma_{\mathbf h_j}$. Consequently the maps
\eqref{eq:single-flag-map} give a surjective $G$-map
\begin{equation}\label{eq:flag-presentation}
 \varphi_n:F_m\twoheadrightarrow J_n(4),
 \qquad e_{j,Bg}\longmapsto g^{-1}\sigma_{\mathbf h_j}.
\end{equation}
Its image is $J_n(4)=\sum_{j=1}^mJ(\mathbf h_j)$.
Its kernel includes relations among generators belonging to different
cyclic images.

Identify $W_m=F_m^*$ with the $\F$-valued functions on
$\{1,\ldots,m\}\times B\backslash G$ by setting
$v_{j,Bg}=v(e_{j,Bg})$.
Thus a coordinate is indexed by a height index $j$ and a coset $Bg$.
The dual group action is
$(\gamma v)_{j,Bg}=v_{j,Bg\gamma}$.
Write $[x^e]f$ for the coefficient of $x^e$ in $f$.
The dual of the composite $F_m\xrightarrow{\varphi_n}J_n(4)
\hookrightarrow\DP_n(4)$ is
\begin{equation}\label{eq:evaluation}
 E_n:P_n(4)\longrightarrow W_m,
 \qquad
 E_n(f)_{j,Bg}=[x_1^{\nu_{h_{j1}}}\cdots x_4^{\nu_{h_{j4}}}]\,gf.
\end{equation}
Indeed, this coordinate is
$\langle g^{-1}\sigma_{\mathbf h_j},f\rangle$.
Put $R_n=\im E_n$. For a subspace $R\subseteq W_m$, define its
annihilator in $F_m$ by
\[
 R^\perp=\{u\in F_m:v(u)=0\text{ for every }v\in R\}.
\]
Duality gives
\begin{equation}\label{eq:R}
 R_n\cong J_n(4)^*,\qquad \ker\varphi_n=R_n^\perp\subseteq F_m.
\end{equation}
\begin{proposition}\label{prop:comparison-criterion}
Let $n,n'$ have ordered height lists
$(\mathbf h_1,\ldots,\mathbf h_m)$ and
$(\mathbf h'_1,\ldots,\mathbf h'_m)$ of the same length.
Identify their presenting modules using the basis vectors $e_{j,Bg}$.
The rule
\[
 g\sigma_{\mathbf h_j}\longmapsto g\sigma_{\mathbf h'_j}
 \qquad(g\in G,\ 1\leq j\leq m)
\]
defines a $G$-isomorphism $J_n(4)\cong J_{n'}(4)$ if and only if
$R_n=R_{n'}$ as subspaces of the common $W_m$.
\end{proposition}
\begin{proof}
The rule descends from the identity of $F_m$ precisely when
$\ker\varphi_n\subseteq\ker\varphi_{n'}$.
It defines an isomorphism with reverse correspondence
$g\sigma_{\mathbf h'_j}\mapsto g\sigma_{\mathbf h_j}$ precisely
when these kernels are equal.
By~\eqref{eq:R}, their equality is equivalent to
$R_n^\perp=R_{n'}^\perp$, and hence to $R_n=R_{n'}$, since
$F_m$ and $W_m$ are finite-dimensional dual spaces.
More explicitly, when the kernels agree, define
\[
 \Theta(\varphi_n(v))=\varphi_{n'}(v)\qquad(v\in F_m).
\]
If $v,w$ represent the same element on the left, then
$v-w$ belongs to the common kernel, so their images on the right
are equal. Surjectivity follows from that of $\varphi_{n'}$.
If $\Theta(\varphi_n(v))=0$, then $v$ is in the same kernel
and $\varphi_n(v)=0$, proving injectivity.
For $\gamma\in G$,
\[
 \Theta(\gamma\varphi_n(v))=\varphi_{n'}(\gamma v)
 =\gamma\varphi_{n'}(v)=\gamma\Theta(\varphi_n(v)).
\]
Thus the map is equivariant. Its formula on the generators of
$F_m$ gives the asserted values on all translated spikes.
\end{proof}

The common presentation records both relations within each cyclic
module and relations involving generators from different cyclic modules.

\begin{lemma}\label{lem:height-lists}
The complete ordered height lists for~\eqref{eq:families} are
\begin{equation}\label{eq:heights}
\begin{array}{c|l}
\text{Degree and range}&\text{Height vectors, in their chosen order}\\\hline
n_{\II}(s,t)&(s+t+1,s,s,0),\ (s+t,s+t,s+1,0)\\
n_{\III}(s,t),\ s\geq2&(s+t,s,0,0),\ (s+t,s-1,s-1,1),\\
 &(s+t-1,s+t-1,s,1)\\
n_{\III}(1,t),\ t\geq3&(t+1,1,0,0),\ (t,t,1,1),\ (t,t-1,t-1,2)\\
n_{\III}(1,2)=8&(3,1,0,0),\ (2,2,1,1)\\
n_{\III}(1,1)=4&(2,1,0,0),\ (1,1,1,1).
\end{array}
\end{equation}
\end{lemma}
\begin{proof}
The defining equation is $\sum_{i=1}^4 2^{h_i}=n+4$.
For the first family,
$n+4=2^{s+t+1}+2^{s+1}+1$ has three nonzero binary digits.
Merging equal powers in any expression as four powers of two
reduces it to this binary expression in exactly one step.
Conversely, the required expression is obtained by splitting one
positive power $2^a$ into $2^{a-1}+2^{a-1}$.
Splitting the largest or the next largest power gives exactly
the two displayed vectors; the summand $1$ cannot be split.

For the second family with $s\geq2$, the three nonzero binary
digits of $n+4=2^{s+t}+2^s+2$ are distinct. Splitting respectively
$2$, $2^s$, or $2^{s+t}$ gives the three displayed vectors.
For $s=1,t\geq3$, start with $n+4=2^{t+1}+4$ and split twice.
Splitting $4$ and then $2$ gives the first vector; splitting each
of the two original summands gives the second; splitting the
largest summand and then one of its halves gives the third.
These exhaust the possibilities, because every split must involve
one of the summands currently present. At $t=2$, the second and
third possibilities coincide after ordering and give only two
vectors. At $t=1$, the only four-summand expressions of $8$
are $4+2+1+1$ and $2+2+2+2$.
\end{proof}

We can now state the objective of the comparison argument. The
result concerns the full $G$-module presentations, including the
specified correspondence of generators.
For either family, write $\mathbf h_j(s,t)$ for the $j$th height
vector in the corresponding row of~\eqref{eq:heights}.

\begin{proposition}\label{prop:stability}
Fix one of the two degree families and let $n$ be its degree at
positive integers $s,t$. Put
\[
 (s_0,t_0)=
 \begin{cases}
 (\min(s,3),\min(t,4)),& n=n_{\II}(s,t),\\
 (1,\min(t,6)),& n=n_{\III}(1,t),\\
 (\min(s,5),\min(t,5)),& n=n_{\III}(s,t),\ s\geq2.
 \end{cases}
\]
Let $n_0$ be the degree of the same family at $(s_0,t_0)$.
There is a $G$-isomorphism
\[
 \Theta_{s,t}:J_n(4)\xrightarrow{\cong}J_{n_0}(4),\qquad
 \Theta_{s,t}(g\sigma_{\mathbf h_j(s,t)})
   =g\sigma_{\mathbf h_j(s_0,t_0)}
\]
for every $g\in G$ and every index $j$ in the corresponding height list.
\end{proposition}

The proof of this proposition is given in Section~\ref{sec:comparisons}.
It establishes equality of the two subspaces $R_n,R_{n_0}$
in $W_m$, so Proposition~\ref{prop:comparison-criterion}
supplies the map and proves its equivariance.

\section{The two degree families}\label{sec:families}
The comparison isomorphisms describe spike submodules for every
positive parameter pair. Their relation to the full Steenrod kernel
depends on the dimensions of the cohit spaces. We give this comparison
separately for the two families, including the columns outside
Theorem~\ref{thm:main}, before considering the transfer.

Write $q_n=\dim Q_n(4)=\dim K_n(4)$ and $j_n=\dim J_n(4)$.
Then $q_n-j_n=\dim(K_n(4)/J_n(4))$ is the codimension of
the spike submodule in the Steenrod kernel.
The values of $q_n$ below are Sum's dimension formulas
\cite[Theorem~1.1]{Sum2014}; see also
\cite[Theorem~4]{Sum2013} and \cite[Theorem~1.4]{Sum2015}.
The values of $j_n$ are ranks of the maps~\eqref{eq:evaluation},
calculated by the recursion of Section~\ref{sec:frobenius}.
Proposition~\ref{prop:stability} extends each representative
calculation to its entire stated parameter range.
The complete table of thirty-eight representative spike
dimensions and its calculation records are included in the
consolidated archive~\cite{Data}.

\subsection{The family \texorpdfstring{$n_{\II}(s,t)$}{nII(s,t)}}
The initial degrees are displayed with rows indexed by $s$ and
columns indexed by $t$:
\[
\begin{array}{c|rrrr}
s\backslash t&1&2&3&4\\\hline
1&9&17&33&65\\
2&21&37&69&133\\
3&45&77&141&269.
\end{array}
\]
Proposition~\ref{prop:stability} identifies the spike modules after
replacing $s$ by $\min(s,3)$ and $t$ by $\min(t,4)$, with the
prescribed values on translated generators.
The column $t=3$ is part of the same comparison theorem.
Table~\ref{tab:dimensions-II} gives the complete dimension comparison.

\begin{table}[htbp]
\centering
\caption{Dimensions in the first family for all positive parameters.
The degree $n_0=n_{\II}(s_0,t_0)$ is the representative from
Proposition~\ref{prop:stability}; the dimensions apply to every
pair in the row.}
\label{tab:dimensions-II}
\begin{tabular}{@{}ccrrrr@{}}
\toprule
$s$&$t$&$n_0$&$q_n$&$j_n$&$q_n-j_n$\\\midrule
1&1&9&46&45&1\\
2&1&21&94&90&4\\
$\geq3$&1&45&105&105&0\\\addlinespace
1&2&17&87&86&1\\
2&2&37&135&135&0\\
$\geq3$&2&77&150&150&0\\\addlinespace
1&3&33&136&135&1\\
2&3&69&180&180&0\\
$\geq3$&3&141&195&195&0\\\addlinespace
1&$\geq4$&65&150&150&0\\
2&$\geq4$&133&195&195&0\\
$\geq3$&$\geq4$&269&210&210&0\\
\bottomrule
\end{tabular}
\end{table}

The inclusions $J_n(4)\subsetneq K_n(4)$ in this family occur
at $n=9,17,21,33$.
At degree $37$, one has $J_{37}(4)=K_{37}(4)$ and
$(K_{37}(4))_G=0$.
At degree $21$, the quotient $K_{21}(4)/J_{21}(4)$ is
four-dimensional and $(K_{21}(4))_G=0$.
Thus a positive codimension of $J_n(4)$ in $K_n(4)$ is
compatible with a zero-dimensional transfer domain.

\subsection{The family \texorpdfstring{$n_{\III}(s,t)$}{nIII(s,t)}}
The corresponding initial degrees are
\[
\begin{array}{c|rrrrrr}
s\backslash t&1&2&3&4&5&6\\\hline
1&4&8&16&32&64&128\\
2&10&18&34&66&130&258\\
3&22&38&70&134&262&518\\
4&46&78&142&270&526&1038\\
5&94&158&286&542&1054&2078.
\end{array}
\]
For $s\geq2$, the comparison of spike modules reduces $s$ to
$\min(s,5)$ for every $t$. For each fixed $s\geq2$, the comparison maps identify the spike
modules for all $t\geq4$ with the prescribed correspondence of
translated generators, by the identities in Section~\ref{sec:comparisons}.
We retain $t_0=5$ for $t\geq5$ in Proposition~\ref{prop:stability}
because $\dim K_{n_{\III}(2,4)}(4)=179$, whereas
$\dim J_{n_{\III}(2,4)}(4)=175$ and
$\dim K_{n_{\III}(2,t)}(4)=175$ for every $t\geq5$.
For $s=1$, the spike dimensions become constant from $t=6$.

\begin{table}[htbp]
\centering
\caption{Dimensions in the second family for all positive parameters.
The representative degree is $n_0=n_{\III}(s_0,t_0)$ as in
Proposition~\ref{prop:stability}.}
\label{tab:dimensions-III}
\begin{tabular}{@{}ccrrrr@{}}
\toprule
$s$&$t$&$n_0$&$q_n$&$j_n$&$q_n-j_n$\\\midrule
1&1&4&21&21&0\\
2&1&10&70&70&0\\
3&1&22&116&115&1\\
4&1&46&164&160&4\\
$\geq5$&1&94&175&175&0\\\addlinespace
1&2&8&55&51&4\\
2&2&18&126&125&1\\
3&2&38&192&191&1\\
4&2&78&240&240&0\\
$\geq5$&2&158&255&255&0\\\addlinespace
1&3&16&73&69&4\\
2&3&34&165&161&4\\
3&3&70&241&240&1\\
4&3&142&285&285&0\\
$\geq5$&3&286&300&300&0\\\addlinespace
1&4&32&95&94&1\\
2&4&66&179&175&4\\
3&4&134&255&255&0\\
4&4&270&300&300&0\\
$\geq5$&4&542&315&315&0\\\addlinespace
1&5&64&115&115&0\\
1&$\geq6$&128&125&125&0\\
2&$\geq5$&130&175&175&0\\
3&$\geq5$&262&255&255&0\\
4&$\geq5$&526&300&300&0\\
$\geq5$&$\geq5$&1054&315&315&0\\
\bottomrule
\end{tabular}
\end{table}

The two tables give the following consequence of the spike
comparisons and Sum's dimension theorem
\cite[Theorem~1.1]{Sum2014}.

\begin{proposition}\label{prop:KJ}
For a degree in either complete family~\eqref{eq:families}, set
\[
 \mathcal E_1=\{9,17,18,22,32,33,38,70\},\qquad
 \mathcal E_4=\{8,16,21,34,46,66\}.
\]
Then
\[
 \dim(K_n(4)/J_n(4))=
 \begin{cases}
 1,&n\in\mathcal E_1,\\
 4,&n\in\mathcal E_4,\\
 0,&n\notin\mathcal E_1\cup\mathcal E_4.
 \end{cases}
\]
Whenever $(s,t)\ne(s_0,t_0)$ in
Proposition~\ref{prop:stability}, one has $J_n(4)=K_n(4)$
and $J_{n_0}(4)=K_{n_0}(4)$.
Consequently $\Theta_{s,t}$ is also an isomorphism between the
full Steenrod kernels in these two degrees. In the parameter ranges of
Theorem~\ref{thm:main}, the proper inclusions occur precisely
at $n=9,17,21,22,46$.
\end{proposition}
\begin{proof}
The inclusion $J_n(4)\subseteq K_n(4)$ gives
\[
 \dim(K_n(4)/J_n(4))=\dim K_n(4)-\dim J_n(4)=q_n-j_n.
\]
Substituting the values in Tables~\ref{tab:dimensions-II}
and~\ref{tab:dimensions-III} gives the asserted two exceptional
sets and their codimensions. All other inclusions have equal
finite dimensions and therefore are equalities.

A degree determines its family and parameter pair uniquely.
Indeed, the first family consists of odd degrees and the second
of even degrees. Within either family, respectively,
\[
 n+3=2^{s+1}(2^t+1),\qquad n+2=2^s(2^t+1).
\]
Since $2^t+1$ is odd, the power of two dividing the left side
determines $s$, and the remaining odd factor determines $t$.

Every exceptional parameter pair is fixed by the reduction in
Proposition~\ref{prop:stability}. In the first family its first
coordinate is at most two and its second at most three.
A nontrivial reduction either gives $s_0=3$ or $t_0=4$, so
neither endpoint is exceptional.
In the second family the exceptional pairs have $s\leq4$ and
$t\leq4$. A nontrivial reduction with $s\geq2$ gives $s_0=5$
or $t_0=5$, and one with $s=1$ gives $t_0=6$.
Thus $J_n(4)=K_n(4)$ and $J_{n_0}(4)=K_{n_0}(4)$ in this case as well.
The map in Proposition~\ref{prop:stability} consequently has
domain and codomain the full kernels in these cases.
Intersecting the exceptional sets with the restrictions in
Theorem~\ref{thm:main} gives its five stated degrees.
\end{proof}

\subsection{The remaining transfer parameters}\label{sec:omitted}
Proposition~\ref{prop:stability} also applies to $t=3$ in the
first family and to $t=2,3,4$ in the second.
These parameters give further proper inclusions of spike submodules
and additional target classes requiring separate transfer arguments.
Table~\ref{tab:known-remaining} records the invariant dimensions
proved by Sum~\cite[Theorems~1.2 and~1.4]{Sum2025}.

\begin{table}[htbp]
\centering
\caption{Invariant dimensions $\dim Q_n(4)^G$ in the
remaining columns, computed by Sum
\cite[Theorems~1.2 and~1.4]{Sum2025}, with the first parameter
in his first family equal to $s+1$.}
\label{tab:known-remaining}
\begin{tabular}{@{}ccccc@{}}
\toprule
$s$&$n_{\II}(s,3)$&$n_{\III}(s,2)$&$n_{\III}(s,3)$&$n_{\III}(s,4)$\\\midrule
1&1&0&0&1\\
2&0&2&1&1\\
3&1&2&2&2\\
4&1&1&1&2\\
$\geq5$&1&2&2&3\\
\bottomrule
\end{tabular}
\end{table}

In Lin's notation, the target has the additional classes
$p_0,p'_0$ at the first-family degrees $33,69$ and
$f_0,e_1,p_1,p'_1,d_1$ at the second-family degrees
$18,38,70,142,32$, respectively
\cite[Theorem~1.3 and~(1.1)]{Lin2008}.
All seven classes have cohomological degree four; their respective
internal degrees are $37,73,22,42,74,146,36$.
Sum's results show that the transfers in these remaining columns
are isomorphisms except at $69$ and $142$; at those two degrees
the cokernels are represented by $p'_0$ and $p'_1$
\cite[pp.~3--4 and Corollary~1.6]{Sum2025}.
The detected classes $p_0,f_0,e_1,p_1,d_1$ lie outside the subspace
spanned by products of four classes $h_i$; multiplicativity places the
transfer image of the coinvariants $(J_n(4))_G$ of the spike submodule
in that subspace.
The products of the classes $h_i$ also satisfy $h_1^2h_4^2=0$
in $\Ext_{\A}^{4,36}(\F,\F)$, corresponding to homological degree $32$.

Applying coinvariants to the inclusion of the spike submodule and its
quotient gives the right exact sequence
\begin{equation}\label{eq:coinvariant-sequence}
 (J_n(4))_G\longrightarrow(K_n(4))_G
 \longrightarrow(K_n(4)/J_n(4))_G\longrightarrow0.
\end{equation}
Thus the cokernel of the first arrow is $(K_n(4)/J_n(4))_G$,
whereas $q_n-j_n$ is the dimension of $K_n(4)/J_n(4)$ before
taking coinvariants.
Injectivity of the first arrow requires a separate argument:
for the ranges of Theorem~\ref{thm:main}, it is proved in
Proposition~\ref{prop:adams-injection}.
The five initial invariant calculations in Section~\ref{sec:transfer}
determine the full transfer domains.

\section{Frobenius decomposition and the dual flag maps}\label{sec:frobenius}
To apply Proposition~\ref{prop:comparison-criterion}, we describe
the image of the dual flag map in a finite-dimensional coordinate space
whose dimension depends only on the number of height vectors.
The free-module decomposition of $P(4)$ over its subalgebra
of squares provides such coordinates. The weight sequence of a target
spike specifies which parity component is taken at each step.
This construction uses the evaluation method of
\cite[Section~3.3, Lemma~3.3 and Proposition~3.4]{Phuc26e};
we give its formulas and proof here for the ordered presentations
of the two families.

\subsection{Parity components of a polynomial}
The Frobenius map of $P(4)$ sends $f$ to $f^2$; its image is
the subalgebra $P(4)^2=\{f^2:f\in P(4)\}$.
The set
\[
 \{0,1\}^4=\{(\delta_1,\delta_2,\delta_3,\delta_4):
                 \delta_i\in\{0,1\}\text{ for }1\leq i\leq4\}
\]
consists of the sixteen ordered four-tuples whose entries are zero or one.
We regard their entries as integers when using them as exponents.
In particular, $x^\delta=\prod_{i=1}^4x_i^{\delta_i}$ and
$|\delta|=\sum_{i=1}^4\delta_i$.
Every polynomial has a unique decomposition
\begin{equation}\label{eq:frobenius-decomposition}
 f=\sum_{\delta\in\{0,1\}^4}x^\delta f_\delta^2,
 \qquad f_\delta\in P(4).
\end{equation}
This follows by writing each exponent vector uniquely as
$2c+\delta$, with $c\in\mathbb Z_{\geq0}^4$ and
$\delta\in\{0,1\}^4$.
Thus $P(4)$ is free over $P(4)^2$ with basis the sixteen
squarefree monomials $x^\delta$, meaning those with every exponent
equal to zero or one. Modulo the ideal generated by
$x_1^2,x_2^2,x_3^2,x_4^2$, they give the usual basis of the exterior
algebra on the four variable classes.
The decomposition is taken with respect to the fixed variables
$x_1,\ldots,x_4$.

Define the \emph{parity component} $C_\delta(f)$ to be $f_\delta$
in~\eqref{eq:frobenius-decomposition}. Equivalently,
\[
 C_\delta\left(\sum_{e\in\mathbb Z_{\geq0}^4}b_ex^e\right)
 =\sum_{c\in\mathbb Z_{\geq0}^4}b_{2c+\delta}x^c.
\]
Here $b_e\in\F$, with only finitely many nonzero coefficients.
It is an $\F$-linear map.
For a homogeneous polynomial of degree $d$, a nonzero component
$C_\delta(f)$ is homogeneous of degree $(d-|\delta|)/2$;
the component is zero if this is not a nonnegative integer.
Write $\lfloor u\rfloor$ for the greatest integer not exceeding $u$.
For arbitrary $f$, a nonzero $C_\delta(f)$ has degree at most
$\lfloor(\deg f-|\delta|)/2\rfloor$.
The component $C_{(1,1,1,1)}$ is the polynomial map inducing
Kameko's squaring map $\kappa$ of the Introduction
\cite{Kameko1990}; see also~\cite[Definition~6.5.1]{WW1}.

For $f,h\in P(4)$ and $c\in\mathbb Z_{\geq0}^4$,
\begin{equation}\label{eq:coefficient-extraction}
 C_\delta(fh^2)=C_\delta(f)h,
 \qquad [x^{2c+\delta}]f=[x^c]C_\delta(f).
\end{equation}
For the first identity, substitute~\eqref{eq:frobenius-decomposition}
for $f$ and observe that
$fh^2=\sum_\delta x^\delta(f_\delta h)^2$; uniqueness gives
the result. The second identity follows directly from the
definition of the component.

\subsection{Partial evaluations and the degree bound}
For each coset $Bg$, fix the matrix representative described in
Appendix~\ref{app:identities}. Every substitution by $g$ in the
polynomial tuples below uses that representative.
Fix first one height vector $\mathbf h$ with weight sequence
$(\omega_0,\omega_1,\ldots)$.
Let $\delta_a\in\{0,1\}^4$ have its first $a$ entries equal
to one and the others zero, for $0\leq a\leq4$.
The binary digits in position $r$ of the four spike exponents form
the vector $\delta_{\omega_r}$.
At a given step, let $e\in\mathbb Z_{\geq0}^4$ be the remaining
exponent vector of the input monomial. Write uniquely $e=\epsilon+2e'$,
where $\epsilon\in\{0,1\}^4$ and $e'\in\mathbb Z_{\geq0}^4$.
Then
\[
 g(x^e)=g(x^\epsilon)g(x^{e'})^2.
\]
Starting with the constant polynomial $p=1$, the step at binary
position $r$ replaces $p$ by
\[
 C_{\delta_{\omega_r}}\bigl(p\,g(x^\epsilon)\bigr),
\]
where $\epsilon$ is the input digit vector at that position.
We call the polynomials obtained by these successive updates
the \emph{partial evaluations}.
Equation~\eqref{eq:coefficient-extraction} shows how each update
removes one binary digit from the target exponents while retaining
a polynomial factor for the higher digits.

These partial evaluations have a uniform degree bound.
If a remaining factor $p$ has degree at most three, then
$\deg(p\,g(x^\epsilon))\leq3+4=7$, and
\[
 \deg C_{\delta_a}(p\,g(x^\epsilon))
 \leq\left\lfloor\frac{7-a}{2}\right\rfloor\leq3
\]
when the result is nonzero. Hence every partial evaluation lies in
\[
 V=P_{\leq3}(4)=\bigoplus_{d=0}^3P_d(4)
  =\Span\{x^c:|c|\leq3\},\qquad
 \dim V=1+4+10+20=35.
\]
Here $P_{\leq3}(4)$ denotes the polynomials of total degree at most three.
The degree bound is independent of the height vector and of the
number of digits processed.

For a finite set $I$, write $V^I$ for the vector space of $I$-indexed
tuples with entries in $V$, with addition and scalar multiplication
defined coordinatewise.
For all flags of this one height vector, the partial evaluations
form a tuple $(p_{Bg})\in V^{B\backslash G}$.
For all $m$ height vectors in~\eqref{eq:flag-presentation}, use
\[
 \mathcal V_m=V^{\{1,\ldots,m\}\times B\backslash G},
 \qquad \dim\mathcal V_m=35\cdot315m.
\]
Its entries are polynomials recording the remaining factors;
the entries of $W_m$ are their eventual scalar evaluations.
Write $[1]f$ for the constant coefficient of a polynomial $f$.
The coordinatewise constant-term map is
\[
 \pi:\mathcal V_m\longrightarrow W_m,
 \qquad(\pi p)_{j,Bg}=[1]p_{j,Bg}.
\]
Let $U_m\subseteq\mathcal V_m$ be the one-dimensional subspace
spanned by the tuple all of whose entries are the constant polynomial $1$.

At position $r$, write
\[
 a_{jr}=\#\{i:h_{ji}>r\},\qquad
 \mathbf a_r=(a_{1r},\ldots,a_{mr}).
\]
Thus $\mathbf a_r$ is the list of weights at that position,
one for each target spike.
For $\mathbf a=(a_1,\ldots,a_m)\in\{0,1,2,3,4\}^m$, whose
entries are integers between zero and four, and
$\epsilon\in\{0,1\}^4$, define the linear map
\begin{equation}\label{eq:operators}
 T_{\mathbf a,\epsilon}:\mathcal V_m\longrightarrow\mathcal V_m,
 \qquad
 (T_{\mathbf a,\epsilon}p)_{j,Bg}
 =C_{\delta_{a_j}}(p_{j,Bg}\,g(x^\epsilon)).
\end{equation}
The degree bound just proved shows that this map is defined on
the indicated finite-dimensional space.
For a subspace $U\subseteq\mathcal V_m$, the image
$T_{\mathbf a,\epsilon}(U)$ consists of all
$T_{\mathbf a,\epsilon}p$ with $p\in U$. Put
\begin{equation}\label{eq:subspace-operation}
 L_{\mathbf a}(U)=\sum_{\epsilon\in\{0,1\}^4}
                   T_{\mathbf a,\epsilon}(U)
 \subseteq\mathcal V_m.
\end{equation}
The operation $L_{\mathbf a}$ assigns to $U$ the span of all
$T_{\mathbf a,\epsilon}p$ with $p\in U$ and $\epsilon\in\{0,1\}^4$.
The same parity vector $\epsilon$ is used in every coordinate:
it is the digit vector of a single input monomial, whose evaluations
determine all coordinates of its image under $E_n$.

\subsection{The recursion for the dual presentation}
For an integer $r\geq0$, write $L_{\mathbf a}^r(U)$ for $r$
successive applications,
with $L_{\mathbf a}^0(U)=U$. Parentheses around subspace arguments
are omitted in products, which are applied from right to left.
For a subspace $S\subseteq\mathcal V_m$, the notation $\pi S$
means its image $\{\pi p:p\in S\}$ in $W_m$.
The following lemma recovers the image of the dual flag map by
these operations.

\begin{lemma}\label{lem:coefficient-recursion}
Let $n>0$ have a nonempty ordered height list
$\mathcal H_n=(\mathbf h_1,\ldots,\mathbf h_m)$, with
$\mathbf h_j=(h_{j1},\ldots,h_{j4})$, and put
$\ell=\max_{1\leq j\leq m,\ 1\leq i\leq4}h_{ji}$ and
\[
 a_{jr}=\#\{i:h_{ji}>r\},
 \qquad
 \mathbf a_r=(a_{1r},\ldots,a_{mr})
 \quad(0\leq r<\ell).
\]
With the height vectors ordered as in~\eqref{eq:evaluation},
\begin{equation}\label{eq:recursion}
 R_n=\pi L_{\mathbf a_{\ell-1}}\cdots
          L_{\mathbf a_1}L_{\mathbf a_0}U_m.
\end{equation}
\end{lemma}

\begin{proof}
Fix $e\in\mathbb Z_{\geq0}^4$, an index $1\leq j\leq m$,
and the chosen representative $g$ of a coset $Bg$.
For $r\geq0$, put
\[
 e^{(r)}
 =\left(\left\lfloor\frac{e_i}{2^r}\right\rfloor\right)_{i=1}^4,
 \qquad
 d_j^{(r)}
 =\left(\left\lfloor\frac{\nu_{h_{ji}}}{2^r}\right\rfloor\right)_{i=1}^4.
\]
For $0\leq r<\ell$, there is a unique $\epsilon_r\in\{0,1\}^4$
such that
\[
 e^{(r)}=\epsilon_r+2e^{(r+1)}.
\]
Also,
\[
 d_j^{(r)}=\delta_{a_{jr}}+2d_j^{(r+1)}.
\]
For the latter identity, the digit in position $r$ of
$2^{h_{ji}}-1$ is one precisely when $h_{ji}>r$.
The weakly decreasing order of the heights places those ones
in the first $a_{jr}$ coordinates.

Start with $p^{(0)}_{j,Bg}=1$ at every coordinate and define
\[
 p^{(r+1)}
 =T_{\mathbf a_r,\epsilon_r}p^{(r)}
 \qquad(0\leq r<\ell).
\]
Since linear substitution commutes with squaring,
\[
 g(x^{e^{(r)}})
 =g(x^{\epsilon_r})\,g(x^{e^{(r+1)}})^2.
\]
Applying~\eqref{eq:coefficient-extraction} gives
\[
 [x^{d_j^{(r)}}]
   \bigl(p^{(r)}_{j,Bg}g(x^{e^{(r)}})\bigr)
  =
 [x^{d_j^{(r+1)}}]
   \bigl(
     C_{\delta_{a_{jr}}}
       (p^{(r)}_{j,Bg}g(x^{\epsilon_r}))
     g(x^{e^{(r+1)}})
   \bigr)
=
 [x^{d_j^{(r+1)}}]
   \bigl(p^{(r+1)}_{j,Bg}g(x^{e^{(r+1)}})\bigr).
\]
This identity preserves the coefficient being computed at each
successive binary position.

After $\ell$ iterations, $d_j^{(\ell)}=0$ for every $j$.
If every $e_i<2^\ell$, then $e^{(\ell)}=0$, and the final expression
is $[1]p^{(\ell)}_{j,Bg}$.
Consequently
\[
 E_n(x^e)=\pi p^{(\ell)}
 \qquad\text{when }|e|=n
 \text{ and }0\leq e_i<2^\ell.
\]
If some $e_i\geq2^\ell$, then $g(x^{e^{(\ell)}})$ is homogeneous
of positive degree. Its product with
$p^{(\ell)}_{j,Bg}$ has zero constant term, so the original
evaluation coefficient is zero at every coordinate.

Finally, the choices
$\epsilon_0,\ldots,\epsilon_{\ell-1}\in\{0,1\}^4$
enumerate all exponent vectors with $0\leq e_i<2^\ell$.
Successive application of the subspace operations $L_{\mathbf a_r}$
takes the span of the corresponding tuples $p^{(\ell)}$.
Homogeneity makes the final coefficient zero when $|e|\neq n$.
The nonzero evaluations therefore span exactly $E_n(P_n(4))$,
which proves~\eqref{eq:recursion}.
\end{proof}

\section{Equivariant comparisons in the two families}\label{sec:comparisons}\label{sec:uniform}
The height lists determine explicit compositions of the operations
$L_{\mathbf a}$. The subspace identities in Appendix~\ref{app:identities}
allow sufficiently large exponents in these compositions to be replaced
by fixed exponents. The resulting equalities of dual images give the
comparison maps of Proposition~\ref{prop:stability}.
We then describe the standard duplication maps and derive partial
flag decompositions that provide theoretical proofs of the
comparisons in specified parameter ranges.

\subsection{Comparisons from finite subspace identities}
\begin{lemma}\label{lem:family-recursions}
Let $R_{\II}(s,t)$ and $R_{\III}(s,t)$ denote the images
$R_n$ of~\eqref{eq:evaluation} for the ordered height lists
in~\eqref{eq:heights}. Then
\begin{align}
 R_{\II}(s,t)
 &=\pi L_{(1,0)}L_{(1,2)}^{t-1}L_{(1,3)}L_{(3,3)}^sU_2,
 &&s,t\geq1,\label{eq:II-operators}\\
 R_{\III}(s,t)
 &=\pi L_{(1,1,0)}L_{(1,1,2)}^{t-1}L_{(2,1,3)}
       L_{(2,3,3)}^{s-2}L_{(2,4,4)}U_3,
 &&s\geq2,\ t\geq1,\label{eq:III-operators}\\
 R_{\III}(1,t)
 &=\pi L_{(1,0,0)}L_{(1,2,1)}L_{(1,2,3)}^{t-3}
       L_{(1,2,4)}L_{(2,4,4)}U_3,
 &&t\geq3.\label{eq:III-one-operators}
\end{align}
The two remaining cases are
\begin{equation}\label{eq:small-recursions}
 R_4=\pi L_{(1,0)}L_{(2,4)}U_2,
 \qquad R_8=\pi L_{(1,0)}L_{(1,2)}L_{(2,4)}U_2.
\end{equation}
Here $R_4=R_{\III}(1,1)$ and $R_8=R_{\III}(1,2)$; both are
subspaces of $W_2$. The operations with two-entry subscripts act on
$\mathcal V_2$, and those with three-entry subscripts act on $\mathcal V_3$.
\end{lemma}
\begin{proof}
For the first height vector in the first family, the weights
are $3$ in positions $0,\ldots,s-1$ and $1$ in positions
$s,\ldots,s+t$. For the second, they are $3$ in positions
$0,\ldots,s$, $2$ in positions $s+1,\ldots,s+t-1$,
and $0$ in position $s+t$.
The interval $s+1,\ldots,s+t-1$ is empty when $t=1$.
Together they give $s$ copies of $(3,3)$, then $(1,3)$,
then $t-1$ copies of $(1,2)$, and finally $(1,0)$.
Lemma~\ref{lem:coefficient-recursion} gives
\eqref{eq:II-operators}.

For the second family with $s\geq2$, the three height vectors
give $(2,4,4)$ in position zero, $(2,3,3)$ in positions
$1,\ldots,s-2$, $(2,1,3)$ in position $s-1$,
$(1,1,2)$ in positions $s,\ldots,s+t-2$, and
$(1,1,0)$ in position $s+t-1$.
The interval $1,\ldots,s-2$ is empty at $s=2$, and the
interval $s,\ldots,s+t-2$ is empty at $t=1$.
This proves~\eqref{eq:III-operators}, including those endpoints.

For $s=1,t\geq3$, the tuples are $(2,4,4)$ in position zero,
$(1,2,4)$ in position one, $(1,2,3)$ in positions
$2,\ldots,t-2$, $(1,2,1)$ in position $t-1$, and
$(1,0,0)$ in position $t$. The middle interval is empty at
$t=3$, giving~\eqref{eq:III-one-operators} also in that case.
The two height lists at degrees $4$ and $8$ give respectively
the sequences $((2,4),(1,0))$ and
$((2,4),(1,2),(1,0))$, proving~\eqref{eq:small-recursions}.
\end{proof}

The subspace identities are given in
Lemma~\ref{lem:subspace-identities}. They are identities in the
spaces $\mathcal V_m$ just defined, or in $W_m$ after applying
$\pi$. Their exact matrix entries and verification are described
in Appendix~\ref{app:identities}. The following argument explains
why those finite identities prove the comparison for every
parameter value.

\begin{proof}[Proof of Proposition~\ref{prop:stability}]
For a subspace $U\subseteq\mathcal V_m$ and an integer $r\geq0$,
the equality
$L_{\mathbf a}^{r+1}(U)=L_{\mathbf a}^{r}(U)$ implies
\begin{equation}\label{eq:iteration-principle}
 L_{\mathbf a}^{p}(U)=L_{\mathbf a}^{r}(U)\qquad(p\geq r).
\end{equation}
In fact, $Y=L_{\mathbf a}^r(U)$ satisfies $L_{\mathbf a}(Y)=Y$;
induction on the nonnegative integer $p-r$ gives the assertion.

For the first family, the first two identities in
\eqref{eq:II-subspace-identities} give
\[
 L_{(1,3)}L_{(3,3)}^sU_2
 =L_{(1,3)}L_{(3,3)}^{\min(s,3)}U_2.
\]
Indeed, for $s\geq4$ the iterates of $L_{(3,3)}$ equal the
fourth iterate, and applying $L_{(1,3)}$ identifies this with
the third; for $s\leq3$ there is no change.
For each of the remaining exponents $1,2,3$ of $L_{(3,3)}$,
the third identity in~\eqref{eq:II-subspace-identities} and
\eqref{eq:iteration-principle} make the iterates of $L_{(1,2)}$
constant from the third iterate. Thus, if $t\geq4$, the
exponent $t-1$ can be replaced by $3$; if $t=1,2,3$, it is
left unchanged. Equation~\eqref{eq:II-operators} therefore gives
\[
 R_{\II}(s,t)=R_{\II}(\min(s,3),\min(t,4))\subseteq W_2.
\]
In particular, this argument includes the column $t=3$.

For the second family with $s\geq2$, let $p=s-2$.
The first two identities in~\eqref{eq:III-subspace-identities}
give
\[
 L_{(2,1,3)}L_{(2,3,3)}^pL_{(2,4,4)}U_3
 =L_{(2,1,3)}L_{(2,3,3)}^{\min(p,3)}L_{(2,4,4)}U_3.
\]
This is the reduction $s\mapsto\min(s,5)$.
The reduced exponent $\min(p,3)$ belongs to $\{0,1,2,3\}$.
For each of these four values, the final identity in
\eqref{eq:III-subspace-identities} makes the iterates of
$L_{(1,1,2)}$ constant from the third iterate.
When $t\leq4$, keep the exponent $t-1$ unchanged.
When $t\geq5$, both exponents $t-1$ and $4$ lie in this
constant range, so the exponent can be replaced by $4$.
Substitution in~\eqref{eq:III-operators} gives
\[
 R_{\III}(s,t)=R_{\III}(\min(s,5),\min(t,5))\subseteq W_3.
\]
The same argument with the third iterate also shows equality
with $R_{\III}(\min(s,5),4)$ for $t\geq4$.
For $t\geq5$, we use $t_0=5$ to match the representative degrees
in Table~\ref{tab:dimensions-III}.

Finally, take $s=1$. For $t\leq6$ there is no parameter change.
For $t\geq7$, the first equality in
\eqref{eq:III-one-subspace-identities} and
\eqref{eq:iteration-principle} reduce the exponent $t-3$
of $L_{(1,2,3)}$ in~\eqref{eq:III-one-operators} to $4$.
The second equality in~\eqref{eq:III-one-subspace-identities}
identifies the images after applying $L_{(1,2,1)}$,
$L_{(1,0,0)}$, and $\pi$, in that order. It therefore replaces
the exponent $4$ by $3$ in the expression for $R_{\III}(1,t)$.
Consequently
\[
 R_{\III}(1,t)=R_{\III}(1,6)\qquad(t\geq6).
\]
In every case, the two height lists have the same number of entries,
indexed in the order fixed in~\eqref{eq:heights}.
The equality $R_n=R_{n_0}$ has been proved in the same $W_m$
with the same coset coordinates.
Proposition~\ref{prop:comparison-criterion} now gives
$\Theta_{s,t}$, its bijectivity and equivariance, and its stated
values on all translated spike generators.
\end{proof}

\subsection{Global duplication in the first family}\label{sec:duplication}
We next relate the comparisons to the duplication and flag maps in
Walker--Wood~\cite[Section~9.5]{WW1},\cite[Chapters~23--24]{WW2}.
The following statements explain which comparisons also follow from
these constructions. The duplication map supplies a map on the full
Steenrod kernels, and the partial flag embeddings give direct
decompositions in the parameter ranges specified below.

For a degree $d$ with $\mu(d)=3$, the dual duplication map of
\cite[Definition~9.5.7 and Proposition~9.5.8(i)]{WW1} restricts to
a $G$-map
\[
 \delta_d^*:K_{2d+3}(4)\longrightarrow K_d(4).
\]
On a divided-power spike whose first two binary columns are equal
to $(1,1,1,0)$, it deletes the first column. Equivalently, it
subtracts one from each positive height and leaves the zero height
unchanged. This formula will suffice for its restriction to the
spike submodules considered here.

\begin{lemma}\label{lem:global-duplication}
For $s,t\geq1$ and $d=n_{\II}(s,t)$, dual duplication restricts to
a surjective $G$-map
\[
 \delta_d^*:J_{n_{\II}(s+1,t)}(4)
       \longrightarrow J_{n_{\II}(s,t)}(4),\qquad
 g\sigma_{\mathbf h_j(s+1,t)}
       \longmapsto g\sigma_{\mathbf h_j(s,t)}\quad(j=1,2).
\]
For $s\geq3$, this map is an isomorphism. Its iterates give the
reduction $s\mapsto3$, with $t$ fixed, used in the proof of
Proposition~\ref{prop:stability}.
\end{lemma}
\begin{proof}
The degree formula gives $n_{\II}(s+1,t)=2d+3$, and
$\mu(d)=3$ as recorded in the Introduction. The two height vectors
at $(s+1,t)$ are
\[
 (s+t+2,s+1,s+1,0),\qquad
 (s+t+1,s+t+1,s+2,0).
\]
In each vector the three positive entries are at least two.
Consequently the first two binary columns of each spike are
$(1,1,1,0)$. Deleting the first column gives, respectively,
\[
 (s+t+1,s,s,0),\qquad (s+t,s+t,s+1,0),
\]
the height vectors at $(s,t)$. Equivariance gives the asserted
formula for all translates. These translates span the target,
so the restriction is surjective.

For $s\geq3$, Table~\ref{tab:dimensions-II}, whose spike
dimensions follow from the recursion and the preceding proof,
gives equal dimensions in the source and target. A surjective
linear map between these finite-dimensional spaces therefore has
zero kernel. Its iterates have the prescribed values on every
translated spike, and thus agree with the corresponding comparison
maps. This last dimension argument uses the independently computed
spike dimensions $j_n$.
\end{proof}

\subsection{Partial flag modules and the relation between the families}
\label{sec:partial-flags}
For $I\subseteq\{1,2,3\}$, a \emph{partial flag of type $I$}
is a nested collection of row-vector subspaces of $\F^4$,
with one subspace of each dimension in $I$. Write
$\operatorname{Fl}_I$ for the permutation module with basis
$e_X$ indexed by these partial flags. Its action is induced by
$\gamma:X\mapsto X\gamma^{-1}$, where this formula is applied
to every subspace of the flag. This convention agrees with the
action on $F$ in Section~\ref{sec:spikes}.
For $g\in G$, let $X_I(g)$ be the partial flag consisting of the
spans of its first $i$ rows for $i\in I$.

We use the three types $I=\{1,3\},\{2,3\},\{1,2\}$.
For each of them, forgetting the missing subspace defines a
surjective $G$-map, and summing over its complete refinements
defines a map in the reverse direction:
\[
 \begin{aligned}
 q_I:F&\longrightarrow\operatorname{Fl}_I,
       &q_I(e_{Bg})&=e_{X_I(g)},\\
 \iota_I:\operatorname{Fl}_I&\longrightarrow F,
       &\iota_I(e_X)&=\sum_{Bg:\,X_I(g)=X}e_{Bg}.
 \end{aligned}
\]
A complete refinement means a complete flag containing all the
subspaces of the partial flag. In each of the three types, there
are exactly three such refinements: the missing subspace amounts
to choosing a line in a two-dimensional quotient space over $\F$.
It follows that $q_I\iota_I=3\operatorname{id}
=\operatorname{id}$ over $\F$, where $\operatorname{id}$ is the identity
map on $\operatorname{Fl}_I$. In particular,
$F=\im\iota_I\oplus\ker q_I$.
Thus $\iota_I$ identifies $\operatorname{Fl}_I$ with the
direct summand of $F$ spanned by these refinement sums, as described in
\cite[Proposition~16.5.1]{WW2}.

Let $\mathbf h$ be a height vector for which the distinct positive
entries of $\omega(\mathbf h)$ form one of these three sets $I$.
The stabilizer $P_I$ of the standard partial flag $X_I(1)$ is the
subgroup of $G$ preserving each of its subspaces; here $1$ denotes
the identity matrix. It is generated by $B$ and the permutation
interchanging the two adjacent coordinates whose intervening dimension
is omitted from $I$~\cite[Proposition~16.4.5]{WW2}. Those coordinates have equal heights. Hence this permutation
fixes $\sigma_{\mathbf h}$, as does $B$, and the flag map factors as
\begin{equation}\label{eq:partial-factorization}
 \varphi_{\mathbf h}=\overline\varphi_{\mathbf h}q_I,
 \qquad
 \overline\varphi_{\mathbf h}:\operatorname{Fl}_I
       \twoheadrightarrow J(\mathbf h),\qquad
 e_{X_I(g)}\longmapsto g^{-1}\sigma_{\mathbf h}.
\end{equation}
The formula is well defined because $P_I$ fixes the spike.
It also gives $\varphi_{\mathbf h}\iota_I
=\overline\varphi_{\mathbf h}$.

Put $n=n(\mathbf h)$ and $\omega=\omega(\mathbf h)$. Let
$\DP_n^{>\omega}(4)$ be the span of the divided-power monomials
of degree $n$ having weight strictly greater than $\omega$ in
the lexicographic order of Section~\ref{sec:spikes}.
The \emph{local projection at $\omega$} is the quotient map
\[
 p_\omega:\DP_n(4)\longrightarrow
             \DP_n(4)/\DP_n^{>\omega}(4).
\]
Its restriction to the span of monomials of weight at least $\omega$
has as its image the local quotient used in
\cite[Theorem~24.3.10]{WW2}. In particular, this is the quotient
through which the image of $J(\mathbf h)$ is considered below.
We use Nam's partial flag embedding theorem~\cite{Nam2008} in
the form of Walker--Wood~\cite[Theorem~24.3.10 and Section~24.6]{WW2}.
In that form, if the distinct positive entries of a weakly
decreasing weight sequence are
$I=\{i_1<\cdots<i_r\}\subseteq\{1,2,3\}$, the composite
$p_\omega\overline\varphi_{\mathbf h}$ is injective provided
that the multiplicity of $i_j$ is at least
$i_{j+1}-i_{j-1}$ for every $j$, where $i_0=0$ and $i_{r+1}=4$.
The identification of this composite with the map in that theorem
uses $\varphi_{\mathbf h}\iota_I
=\overline\varphi_{\mathbf h}$ above.
When these inequalities hold,
$\overline\varphi_{\mathbf h}$ is an isomorphism onto its image
$J(\mathbf h)$, and
\begin{equation}\label{eq:partial-kernel}
 \ker\varphi_{\mathbf h}=\ker q_I,
 \qquad
 J(\mathbf h)\cap\DP_n^{>\omega}(4)=0.
\end{equation}
Indeed, the first conclusion follows from injectivity of the
composite. If an element of $J(\mathbf h)$ lies in the displayed
intersection, lift it along $\overline\varphi_{\mathbf h}$;
its image under the injective composite is zero, so the lift,
and hence the element, is zero. These observations connect the
local embedding theorem to the global submodules used here.

\begin{proposition}\label{prop:partial-II}
If $s\geq3$ and $t\geq4$, there is a $G$-isomorphism
\[
 J_{n_{\II}(s,t)}(4)
 \cong\operatorname{Fl}_{\{1,3\}}
            \oplus\operatorname{Fl}_{\{2,3\}}.
\]
The maps from the two summands send the standard partial flags
to $\sigma_{\mathbf h_1(s,t)}$ and
$\sigma_{\mathbf h_2(s,t)}$, respectively.
They give the comparison of Proposition~\ref{prop:stability}
throughout this range, with the prescribed values on translated spikes.
\end{proposition}
\begin{proof}
Put $n=n_{\II}(s,t)$. The weight sequences of the two spikes are
\[
 \begin{aligned}
 \omega_1&=(\underbrace{3,\ldots,3}_{s},
                 \underbrace{1,\ldots,1}_{t+1},0,\ldots),\\
 \omega_2&=(\underbrace{3,\ldots,3}_{s+1},
                 \underbrace{2,\ldots,2}_{t-1},0,\ldots).
 \end{aligned}
\]
For $I=\{1,3\}$, the required multiplicities of $1$ and $3$
are both three. The actual multiplicities are $t+1$ and $s$.
For $I=\{2,3\}$, the required multiplicities of $2$ and $3$
are three and two, whereas the actual ones are $t-1$ and $s+1$.
All four inequalities hold under the stated hypotheses.
Thus the two cyclic modules are isomorphic to the indicated
partial flag modules and satisfy~\eqref{eq:partial-kernel}.

There are $15\cdot7=105$ flags of type $\{1,3\}$:
$\F^4$ has $15$ lines, and the hyperplanes containing a fixed
line correspond to the $7$ hyperplanes in its three-dimensional
quotient. There are
$(15\cdot14)/(3\cdot2)=35$ two-dimensional subspaces, by counting
their ordered bases. Each is contained in three hyperplanes,
so there are also $35\cdot3=105$ flags of type $\{2,3\}$.

After the first $s$ entries, the next entry of $\omega_1$ is
$1$ and that of $\omega_2$ is $3$. Hence $\omega_1<\omega_2$.
Every monomial in $J(\mathbf h_2)$ has weight at least
$\omega_2$ by~\cite[Proposition~23.1.4]{WW2}, and thus belongs
to $\DP_n^{>\omega_1}(4)$. Equation~\eqref{eq:partial-kernel}
for $J(\mathbf h_1)$ proves that the two cyclic submodules
have zero intersection. Their sum is all of $J_n(4)$ by
Lemma~\ref{lem:height-lists}, proving the direct sum.

For each parameter pair in this range, the kernel of the
presentation from $F\oplus F$ is therefore
$\ker q_{\{1,3\}}\oplus\ker q_{\{2,3\}}$.
This is the same subspace for every pair. The induced comparison
sends each translated spike to the corresponding translated
spike, which proves the final assertion. In particular,
$\dim J_n(4)=210$ here by the partial flag count alone.
\end{proof}

\begin{proposition}\label{prop:partial-III}
If $s\geq3$ and $t\geq2$, put $d=n_{\II}(s-2,t)$.
Then $2d+4=n_{\III}(s,t)$ and the dual up Kameko map
of~\eqref{eq:dual-up-kameko} gives
\[
 J_{n_{\III}(s,t)}(4)
   =J(\mathbf h_1(s,t))\oplus\Phi_{4,d}J_d(4),
 \qquad J(\mathbf h_1(s,t))\cong\operatorname{Fl}_{\{1,2\}}.
\]
If also $s\geq5$ and $t\geq4$, this becomes
\[
 J_{n_{\III}(s,t)}(4)
 \cong\operatorname{Fl}_{\{1,2\}}
       \oplus\operatorname{Fl}_{\{1,3\}}
       \oplus\operatorname{Fl}_{\{2,3\}}.
\]
The three summands map to the cyclic submodules of the three
height vectors in their listed order. The resulting comparisons
have the prescribed values on translated spike generators.
\end{proposition}
\begin{proof}
Put $n=n_{\III}(s,t)$. The formula $d=2^{s+t-1}+2^{s-1}-3$ gives
$2d+4=n$.
The injective $G$-map $\Phi_{4,d}$ adds one to each height.
It therefore takes the two height vectors of the first family
at $(s-2,t)$ to
\[
 (s+t,s-1,s-1,1),\qquad(s+t-1,s+t-1,s,1),
\]
the second and third height vectors of the second family.
By equivariance, its image on $J_d(4)$ is exactly the sum of
these two cyclic submodules, with all their relations retained.

The first height vector $(s+t,s,0,0)$ has weight
\[
 \omega_1=(\underbrace{2,\ldots,2}_{s},
                 \underbrace{1,\ldots,1}_{t},0,\ldots).
\]
For type $\{1,2\}$, the multiplicity bounds are two for $1$
and three for $2$. They hold because $t\geq2$ and $s\geq3$.
Thus~\eqref{eq:partial-kernel} identifies the first cyclic module
with $\operatorname{Fl}_{\{1,2\}}$ and makes its local projection
injective. This partial flag module has dimension $35\cdot3=105$,
since a two-dimensional subspace contains three lines.
Every monomial in $\Phi_{4,d}J_d(4)$ has all four exponents odd,
and therefore has first weight entry $4$. Since the first entry
of $\omega_1$ is $2$, this image lies in
$\DP_n^{>\omega_1}(4)$ and has zero intersection with the first
cyclic module. The three height vectors exhaust the spikes, so
their sum proves the first decomposition.

For $s\geq5$ and $t\geq4$, Proposition~\ref{prop:partial-II}
applies to $J_d(4)$ because $s-2\geq3$.
Composing its two partial flag maps with $\Phi_{4,d}$ gives the
second decomposition, and each composite takes a standard partial
flag to the stated spike. The kernel of the presentation from
$F^{\oplus3}$ is then
$\ker q_{\{1,2\}}\oplus\ker q_{\{1,3\}}\oplus\ker q_{\{2,3\}}$,
independent of the parameters. This proves the generator
correspondence and gives $\dim J_n(4)=315$ throughout this range.
\end{proof}

Propositions~\ref{prop:partial-II}--\ref{prop:partial-III} are
consequences of the partial flag embedding theorem and the dual
Kameko map. Their proofs use neither finite rank computations nor
cohit-dimension formulas. The duplication lemma supplies the
specified map for reduction in $s$ in the first family; its
isomorphism assertion uses the spike dimensions already established.
For all parameter values, including the boundary rows and columns,
the preceding recursion proof supplies the simultaneous presentation
and the relations between its cyclic images.

\section{The fourth algebraic transfer}\label{sec:transfer}
In this section, we prove Theorem~\ref{thm:main}, the main theorem
stated in the Introduction. Within its parameter ranges,
Proposition~\ref{prop:KJ} identifies the Steenrod kernel $K_n(4)$
with the spike submodule $J_n(4)$ except at $n=9,17,21,22,46$.
We first determine the invariant dimensions in these five degrees.
Relations among spike
coinvariants then identify the transfer on the remaining domains,
and Lin's cohomology calculation determines the cokernels.

\subsection{Invariant dimensions in the five exceptional degrees}\label{sec:direct}
The required finite-dimensional linear algebra can be specified
directly in terms of the Steenrod action and four generators of $G$.
Choose the degree-$n$ monomials as coordinates on $P_n(4)$.
The images
\[
 \Sq^{2^j}(x^e),\qquad j\geq0,\quad 2^j\leq n/2,\quad |e|=n-2^j,
\]
span $H_n(4)$. Indeed, the squares $\Sq^{2^j}$ generate $\A$
\cite[Proposition~3.2.4]{WW1},
so every positive operation has its image in the span of their
images. The instability condition says that $\Sq^r f=0$ for
a homogeneous polynomial $f$ of degree less than $r$.
It gives $\Sq^{2^j}P_{n-2^j}(4)=0$ when $2^j>n/2$
\cite[Proposition~1.1.10]{WW1}.
The coordinates of the displayed images, calculated by
\eqref{eq:cartan}, are the columns of the matrix whose image is $H_n(4)$.

For $f\in P_n(4)$, write $[f]=f+H_n(4)$ for its cohit class.
Choose monomials $b_1,\ldots,b_q$ whose classes form a basis
of $Q_n(4)$, where $q=\dim Q_n(4)$.
Let $\rho_i$ interchange $x_i,x_{i+1}$ for $i=1,2,3$, and
let $\tau x_4=x_4+x_3$, with the other variables fixed.
An elementary transvection replaces one variable $x_i$ by $x_i+x_j$
with $i\ne j$ and fixes all other variables.
The permutations generated by the $\rho_i$ and their conjugates
of $\tau$ give all elementary transvections, so the four elements
$\rho_1,\rho_2,\rho_3,\tau$ generate $G$.
Substitutions commute with Steenrod squares and hence act on $Q_n(4)$.
For each $g$ among these four generators, define the $q\times q$
matrix $M_g$ over $\F$ by
\[
 [gb_j]=\sum_{i=1}^q(M_g)_{ij}[b_i]\qquad(1\leq j\leq q).
\]
Thus $M_g$ acts on column coordinate vectors.
Writing $I_q$ for the $q\times q$ identity matrix, we obtain
\begin{equation}\label{eq:fixed}
 Q_n(4)^G\cong\ker D_n,
 \qquad D_n=
 \begin{pmatrix}
 M_{\rho_1}-I_q\\M_{\rho_2}-I_q\\M_{\rho_3}-I_q\\M_\tau-I_q
 \end{pmatrix}.
\end{equation}
The matrix $D_n$ has $4q$ rows and $q$ columns.
Its equations $D_nv=0$ express $[gf]=[f]$ in $Q_n(4)$
for all four generators, where $v$ is the coordinate vector of $[f]$.

\begin{lemma}\label{lem:initial-invariants}
The quotient dimensions and the ranks in~\eqref{eq:fixed}
are as follows:
\begin{equation}\label{eq:initial-invariants}
\begin{array}{c|rrrrr}
n&9&17&21&22&46\\\hline
\dim Q_n(4)&46&87&94&116&164\\
\rank D_n&45&86&94&115&164\\
\dim Q_n(4)^G&1&1&0&1&0.
\end{array}
\end{equation}
\end{lemma}
\begin{proof}
Form the Steenrod-image matrix from~\eqref{eq:cartan}, take
the quotient, and form the four induced substitution matrices
as specified above. Exact row reduction over $\F$ gives the displayed
values of $q=\dim Q_n(4)$ and $\rank D_n$.
Equation~\eqref{eq:fixed} then gives
$\dim Q_n(4)^G=q-\rank D_n$.
The quotient bases, the four matrices, and the invariant vectors
are supplied in~\cite{Data}. Rerunning the original programs with
the full degree-$n$ monomial basis and every indicated
Steenrod-image column gives the same five quotient dimensions and
ranks. The output of these reruns is included in the same archive.
\end{proof}

\subsection{Relations among spike coinvariants}\label{sec:relations}
The spike relations are most naturally stated with rank and
homological degree retained. Set
\[
 \mathcal J_{k,n}=(J_n(k))_{\GL_k},\qquad
 \mathcal J=\bigoplus_{\substack{k\geq1\\n\geq0}}\mathcal J_{k,n}.
\]
The two indices are rank $k$ and homological degree $n$;
$[z]$ denotes the coinvariant class of $z\in J_n(k)$.

For $z\in J_p(k)$ and $w\in J_q(l)$, their \emph{external product}
$z\times w$ uses disjoint variable sets: on basis elements it is
\[
 a^{(e)}\times a^{(f)}
 =a_1^{(e_1)}\cdots a_k^{(e_k)}
  a_{k+1}^{(f_1)}\cdots a_{k+l}^{(f_l)},
 \qquad |e|=p,\quad |f|=q,
\]
and it extends bilinearly.
Products of translated
spikes are translates of product spikes under the block-diagonal
embedding $\GL_k\times\GL_l\subseteq\GL_{k+l}$. This gives
$z\times w\in J_{p+q}(k+l)$, and the same embedding preserves
coinvariant relations in either factor. Hence
\[
 \mathcal J_{k,p}\otimes\mathcal J_{l,q}
 \longrightarrow\mathcal J_{k+l,p+q},\qquad
 [z]\otimes[w]\longmapsto[z\times w].
\]
We write this product by juxtaposition. It is associative, and
interchanging the variable blocks is a group action, so the
product is commutative on coinvariants. Thus $\mathcal J$ has a
commutative associative multiplication that adds ranks and
homological degrees. Define
\[
 u_i=[a_1^{(2^i-1)}]\in\mathcal J_{1,2^i-1}
       =(J_{2^i-1}(1))_{\GL_1}\qquad(i\geq0).
\]
In particular, $u_0$ has rank one and degree zero; a product
$u_i u_j$ has bidegree $(2,2^i+2^j-2)$.

The equivariant dual up Kameko map $\Phi_{k,n}$
of~\eqref{eq:dual-up-kameko} preserves the spike submodules.
It consequently induces
\[
 \Phi:\mathcal J_{k,n}\longrightarrow\mathcal J_{k,2n+k},\qquad
 [z]\longmapsto[\Phi_{k,n}(z)],
\]
since $\Phi_{k,n}(gz-z)=g\Phi_{k,n}(z)-\Phi_{k,n}(z)$.
We also write $\Phi$ for the map on their direct sum $\mathcal J$.
The monomial formula gives
\[
 \Phi_{k+l,p+q}(z\times w)
 =\Phi_{k,p}(z)\times\Phi_{l,q}(w),
\]
so the induced map is multiplicative for the external product.
For $r\geq0$, define its iterates by
$\Phi^0=\operatorname{id}_{\mathcal J}$ and
$\Phi^{r+1}=\Phi\circ\Phi^r$. They satisfy
\[
 \Phi^r:\mathcal J_{k,n}\longrightarrow
 \mathcal J_{k,\,2^r n+(2^r-1)k},\qquad
 \Phi^r(u_i)=u_{i+r}.
\]
Applying $\Phi^i$ therefore shifts every index in a relation
among the $u_j$ by $i$.

\begin{lemma}\label{lem:source-relations}
For every $i\geq0$, the following identities hold in $\mathcal J$:
\begin{equation}\label{eq:source-relations}
 u_i u_{i+1}=0,\qquad u_i u_{i+2}^2=0,\qquad
 u_{i+1}^3=u_i^2u_{i+2}.
\end{equation}
\end{lemma}
\begin{proof}
Since $\Phi^i$ is linear and multiplicative and sends $u_j$ to
$u_{i+j}$, it is enough to prove the three relations at $i=0$.
For the first relation, work in $\DP(2)$ with $a=a_1$ and $b=a_2$.
Let $\tau_{ab}\in\GL_2$ act by $\tau_{ab}a=a+b$ and
$\tau_{ab}b=b$. Then $(\tau_{ab}-\operatorname{id})a=b$ in $J_1(2)$,
where $\operatorname{id}$ is the identity map.
Since $[b]=u_0u_1$, passing to coinvariants proves $u_0u_1=0$.

For the remaining relations, work in $\DP(3)$ with
$a=a_1$, $b=a_2$, and $c=a_3$. Put
\[
 \begin{aligned}
 z_{ab}&=a^{(2)}b+ab^{(2)}=(a+b)^{(3)}+a^{(3)}+b^{(3)},\\
 z_{ac}&=a^{(2)}c+ac^{(2)}=(a+c)^{(3)}+a^{(3)}+c^{(3)}.
 \end{aligned}
\]
These are sums of translates of spikes.
Let $\tau_{ac}\in\GL_3$ send $a$ to $a+c$ and fix $b,c$,
and let $\sigma_{ab}\in\GL_3$ interchange $a,b$ and fix $c$.
Using the divided-power multiplication formula gives
\[
 \begin{aligned}
 (\tau_{ac}-\operatorname{id})(a^{(3)}+z_{ab})
   &=a^{(2)}c+ac^{(2)}+c^{(3)}+abc+bc^{(2)}+b^{(2)}c,\\
 (\sigma_{ab}-\operatorname{id})z_{ac}
   &=b^{(2)}c+bc^{(2)}+a^{(2)}c+ac^{(2)}.
 \end{aligned}
\]
The four repeated monomials cancel over $\F$, so
\[
 (\tau_{ac}-\operatorname{id})(a^{(3)}+z_{ab})
 +(\sigma_{ab}-\operatorname{id})z_{ac}=c^{(3)}+abc.
\]
The arguments $a^{(3)}+z_{ab}$ and $z_{ac}$ belong to $J_3(3)$.
Both differences therefore vanish in its $\GL_3$-coinvariant quotient,
giving
$0=u_0^2u_2+u_1^3$, and hence $u_1^3=u_0^2u_2$.
Finally, divided-power multiplication gives the following identity
in $J_6(3)\subseteq\DP_6(3)$:
\begin{align*}
 0={}&a^{(3)}b^{(3)}+a^{(3)}c^{(3)}+b^{(3)}c^{(3)}
       +a^{(3)}(b+c)^{(3)}+b^{(3)}(a+c)^{(3)}\\
 &+c^{(3)}(a+b)^{(3)}+(a+b)^{(3)}(a+c)^{(3)}.
\end{align*}
For an explicit verification, the final term expands as
\[
 \begin{aligned}
 (a+b)^{(3)}(a+c)^{(3)}
   ={}&a^{(3)}b^{(3)}+a^{(3)}c^{(3)}+b^{(3)}c^{(3)}\\
     &+a^{(3)}b^{(2)}c+a^{(3)}bc^{(2)}
       +a^{(2)}b^{(3)}c+ab^{(3)}c^{(2)}\\
     &+a^{(2)}bc^{(3)}+ab^{(2)}c^{(3)}.
 \end{aligned}
\]
Expansion of the preceding six terms gives exactly the same nine
monomials: each of $a^{(3)}b^{(3)}$, $a^{(3)}c^{(3)}$ and
$b^{(3)}c^{(3)}$ occurs three times, and each of the six
remaining monomials occurs once. Their sum with the final term
is therefore zero over $\F$.
The seven terms of the original identity are translates of
$a^{(3)}b^{(3)}$, indexed by the seven two-dimensional linear
subspaces of $\F^3$.
Passing to coinvariants therefore gives
$0=7u_0u_2^2=u_0u_2^2$, since $\F$ has characteristic two.
\end{proof}

The \emph{Adams subalgebra} $\mathscr H$ is the subalgebra of
$\Ext_{\A}^{*,*}(\F,\F)$ generated by the classes $h_i$, $i\geq0$,
introduced in Section~\ref{sec:intro}. Its homogeneous
components are
\[
 \mathscr H^{a,b}=\mathscr H\cap\Ext_{\A}^{a,b}(\F,\F),
\]
where $a$ is cohomological degree and $b$ is internal degree.
In the bidegree relevant to $\Tr_{4,n}$,
\[
 \mathscr H^{4,n+4}
 =\operatorname{span}_{\F}
 \left\{h_{i_1}h_{i_2}h_{i_3}h_{i_4}\ \middle|\
 i_1,i_2,i_3,i_4\geq0,\quad\sum_{r=1}^4 2^{i_r}=n+4\right\}.
\]
These products span the component; their relations are specified
next.

Through cohomological degree four, the defining relations of
$\mathscr H$ are
\begin{equation}\label{eq:adams}
 h_i h_{i+1}=0,\quad h_i h_{i+2}^2=0,\quad
 h_{i+1}^3=h_i^2h_{i+2},\quad h_i^2h_{i+3}^2=0
\end{equation}
for $i\geq0$, together with their consequences under multiplication
in these cohomological degrees \cite[Theorems~1.2 and~1.3(1)]{Lin2008}.
The last relation already has cohomological degree four.
Multiplying it by any class of positive cohomological degree
produces a relation of degree greater than four; hence it contributes
to degree four only in its own internal degree $18\cdot2^i$.
This degree
never equals $n+4$ for a degree $n$ in Theorem~\ref{thm:main}.
For the first family, $n+4$ is odd. For the second family with $s\geq2$, it has
three nonzero binary digits, whereas $18\cdot2^i$ has two.
For $s=1$, equality $2^{t+1}+4=18\cdot2^i$ forces $t=4$;
this value is excluded. Thus only the first three relations are applicable
in the degrees considered here.

\begin{proposition}\label{prop:adams-injection}
Let $n$ satisfy the hypotheses of Theorem~\ref{thm:main}.
The natural map $(J_n(4))_{\GL_4}\longrightarrow(K_n(4))_{\GL_4}$ is injective,
and its composite with $\Tr_{4,n}$ is an isomorphism onto
$\mathscr H^{4,n+4}$. In particular, $(J_n(4))_{\GL_4}=0$ for
$n=9,17,21,22,46$.
\end{proposition}
\begin{proof}
Let $U_i$ for $i\geq0$ be formal commuting indeterminates of
bidegree $(1,\nu_i)$. The first degree counts the number of factors;
the second is the sum of their homological degrees $\nu_i=2^i-1$.
Let $I$ be the homogeneous ideal of $\F[U_0,U_1,\ldots]$ generated by
\[
 U_iU_{i+1},\qquad U_iU_{i+2}^2,\qquad
 U_{i+1}^3+U_i^2U_{i+2}\qquad(i\geq0).
\]
Define
\[
 F_{4,n}=\bigl(\F[U_0,U_1,\ldots]/I\bigr)_{(4,n)}.
\]
Thus $F_{4,n}$ is spanned by the classes of formal monomials
$U_{i_1}U_{i_2}U_{i_3}U_{i_4}$ with $\sum_{r=1}^4\nu_{i_r}=n$.
By Lemma~\ref{lem:source-relations}, the assignment
$U_{i_1}U_{i_2}U_{i_3}U_{i_4}\mapsto u_{i_1}u_{i_2}u_{i_3}u_{i_4}$
descends to $F_{4,n}$. Since translates have the same class
in coinvariants, it induces a surjection
$F_{4,n}\twoheadrightarrow(J_n(4))_{\GL_4}$.
Singer's rank-one calculation and multiplicativity send $u_i$ to $h_i$
\cite{Singer1989}. Therefore the composite
\[
 F_{4,n}\twoheadrightarrow(J_n(4))_{\GL_4}
    \longrightarrow(K_n(4))_{\GL_4}\xrightarrow{\Tr_{4,n}}
       \Ext_{\A}^{4,n+4}(\F,\F)
\]
sends a monomial to the same product of $h$-classes. By
\eqref{eq:adams} and the degree check following it, this composite
identifies $F_{4,n}$ with $\mathscr H^{4,n+4}$.
Since the first arrow is surjective and the composite is injective,
the first arrow is an isomorphism, and the composite of the last
two arrows is injective on $(J_n(4))_{\GL_4}$.
Consequently the natural map on coinvariants is injective,
and its transfer image is $\mathscr H^{4,n+4}$, as asserted.

The height lists in~\eqref{eq:heights} give the first-family monomials
$u_0u_s^2u_{s+t+1}$ and $u_0u_{s+1}u_{s+t}^2$.
At $t=1$ the cubic relation gives
$u_0u_s^2u_{s+2}=u_0u_{s+1}^3$; at $t=2$ the second
monomial contains $u_{s+1}u_{s+2}=0$ and vanishes.
The first vanishes when $s=1$ because it contains $u_0u_1$,
and when $s=2$ because it contains $u_0u_2^2$.
For the second family with $s\geq2$, the monomials are
\[
 u_0^2u_su_{s+t},\qquad u_1u_{s-1}^2u_{s+t},\qquad
 u_1u_su_{s+t-1}^2.
\]
At $t=1$ the first vanishes and the other two are equal, with
common class $u_1u_s^3=u_1u_{s-1}^2u_{s+1}$.
At $s=2$, the expression $u_1u_2^3$ contains $u_1u_2=0$.
At $s=3$, the expression $u_1u_2^2u_4$ contains the same zero product.
At $s=4$, the expression $u_1u_3^2u_5$ contains $u_1u_3^2=0$.
Thus the common class vanishes for all three values of $s$.
At $t\geq5$, the first two are equal
when $s=2$ and the third then vanishes; the second vanishes when
$s=3,4$. For $s=1,t\geq5$, the three monomials are
$u_0^2u_1u_{t+1}$, $u_1^2u_t^2$, and $u_2u_{t-1}^2u_t$;
only the middle monomial survives. At $n=4$, the first monomial
$u_0^2u_1u_2$ vanishes because $u_1u_2=0$, and the cubic relation gives
\[
 u_1^4=u_1(u_0^2u_2)=u_0^2u_1u_2=0.
\]
These reductions give the components
$\mathscr H^{4,n+4}$ recorded in
Tables~\ref{tab:II}--\ref{tab:III}; in particular the five
exceptional degrees have $\mathscr H^{4,n+4}=0$.
\end{proof}

\subsection{The target outside the Adams subalgebra}
\begin{lemma}\label{lem:nonadams-targets}
Let $n$ satisfy the hypotheses of Theorem~\ref{thm:main}.
The quotient
\[
 \Ext_{\A}^{4,n+4}(\F,\F)/\mathscr H^{4,n+4}
\]
is zero unless $n=9,17,22,256$. In these four
degrees it is one-dimensional, with respective bases represented
by $h_1c_0,e_0,h_2c_1,D_3(2)$.
\end{lemma}
\begin{proof}
In the cohomology algebra, an \emph{indecomposable class} means
a class that is nonzero modulo the span of products of two
classes of positive cohomological degree.
By Lin's Theorems~1.2 and~1.3 \cite{Lin2008}, the quotient of
$\Ext_{\A}^{4,*}(\F,\F)$ by the subspace of products of four
$h$-classes has a basis represented by the nonzero products
$h_jc_i$, with $i,j\geq0$, and the indecomposable classes
$d_i,e_i,f_i,g_{i+1},p_i,D_3(i),p'_i$ for $i\geq0$.
The product $h_jc_i$ has internal degree $2^j+11\cdot2^i$ and
is zero precisely for $j=i-1,i,i+2,i+3$, when these indices are
nonnegative. Lin's seven indecomposable families in cohomological
degree four have the following internal degrees:
\[
\begin{array}{c|ccccccc}
\text{class}&d_i&e_i&f_i&g_{i+1}&p_i&D_3(i)&p'_i\\
\text{internal degree}
&18\cdot2^i&21\cdot2^i&22\cdot2^i&24\cdot2^i
&37\cdot2^i&65\cdot2^i&73\cdot2^i
\end{array}
\qquad(i\geq0).
\]

Write $N=n+4$ for the internal degree of the transfer target.
For the first family, $N=2^{s+t+1}+2^{s+1}+1$, which is odd. If
$N=2^j+11\cdot2^i$, exactly one of $i,j$ is zero.
For $j=0$, the equation would give
$11\cdot2^i=2^{s+1}(2^t+1)$, hence $11=2^t+1$, which is
impossible. For $i=0$, it becomes
$2^j+10=2^{s+1}(2^t+1)$. An exponent $j\geq2$ makes the
left side divisible by $2$ but not by $4$, whereas the right
side is divisible by $4$. Thus $j=1$, and the equation gives
$s=t=1$. This is $h_1c_0$ in degree $n=9$.
Among the indecomposable families, oddness forces $i=0$ and
leaves the internal degrees $21,37,65,73$. Subtracting one and comparing
with $2^{s+1}(2^t+1)$ gives, respectively,
$(s,t)=(1,2),(1,3)$, no solution, and $(2,3)$.
The two cases with $t=3$ are excluded. Consequently only $e_0$
in degree $n=17$ remains.

For the second family with $s\geq2$, the internal degree
$N=2^{s+t}+2^s+2$ is divisible by $2$ but not by $4$.
If $N=2^j+11\cdot2^i$, equality $i=j$ is impossible, since
then $N=12\cdot2^i$. Hence $\min(i,j)=1$.
The case $j=1<i$ again gives $11=2^t+1$ and is impossible.
In the remaining case $i=1<j$, the equation is
\[
 2^j+20=2^s(2^t+1).
\]
For $j=2$ it gives $(s,t)=(3,1)$, hence $h_2c_1$ at $n=22$.
For $j\geq3$, its left side is divisible by $4$ but not by $8$, so
$s=2$ and $2^t=2^{j-2}+4$.
At $j=3$ the right side is $6$, and at $j=4$ it is $8$.
For $j\geq5$ it is $4(2^{j-4}+1)$, with an odd factor greater
than one, so it is not a power of two.
Thus the only solution is $j=4,t=3$, which is excluded and
also represents the zero product $h_4c_1$.

For the indecomposable classes, the condition $N\equiv2\pmod4$
leaves the candidate internal degrees
$18,42,22,74,130,146$, corresponding respectively to
$d_0,e_1,f_0,p_1,D_3(1),p'_1$.
Subtracting $2$ gives $16,40,20,72,128,144$.
The values $16$ and $128$ are powers of two, so cannot equal
$2^s(2^t+1)$ with $t\geq1$.
The other values are $2^3\cdot5$, $2^2\cdot5$,
$2^3\cdot9$ and $2^4\cdot9$, giving respectively
$(s,t)=(3,2),(2,2),(3,3),(4,3)$.
All four pairs are excluded.

Finally, for $s=1$ the permitted parameters are $t=1$ or
$t\geq5$, and $N=2^{t+1}+4$. At $t=1$, $N=8$ is below
the internal degree of every product $h_jc_i$ and every
indecomposable class. For $t\geq5$, $N$ is divisible by $4$
but not by $8$. Equal indices in $2^j+11\cdot2^i$ would force
$i=j=0$ and $N=12$, which does not occur. For unequal indices,
the smaller index must be $2$. The case $j=2<i$ would make
$11\cdot2^i$ a power of two. The case $i=2<j$ would give
$2^{t+1}=2^j+40$. If $j=3$, the right side is $48$.
If $j\geq4$, it is $8(2^{j-3}+5)$, with an odd factor greater
than one. Neither expression is a power of two.
Thus there are no products $h_jc_i$ in these degrees.
Among the indecomposable families, the requirement that $N$ be
divisible by $4$ but not by $8$ leaves the possible internal
degrees $36,84,44,148,260,292$.
Of these, only $260$ has the form $2^{t+1}+4$ with $t\geq5$;
it gives $t=7$ and the class $D_3(2)$ at $n=256$.
The basis assertion in~\cite[Theorem~1.3]{Lin2008} shows that
the remaining class is nonzero in each of the four quotients.
\end{proof}

\subsection{Completion of the transfer calculation}
We now exhibit classes mapping to $h_1c_0$ and $h_2c_1$.
The remaining image class $e_0$ outside the Adams subalgebra
is detected by H\`a's theorem~\cite[Theorem~1.1(ii)]{Ha2007}.

The external products described in Section~\ref{sec:relations}
are also defined on the coinvariants of the full Steenrod kernels.
In particular, they induce
\[
 (K_p(1))_{\GL_1}\otimes (K_q(3))_{\GL_3}
 \longrightarrow (K_{p+q}(4))_{\GL_4}.
\]
By Boardman's rank-three isomorphism~\cite{Boardman1993}, there
exist classes
\[
 \gamma_0\in(K_8(3))_{\GL_3},
 \qquad
 \gamma_1\in(K_{19}(3))_{\GL_3}
\]
such that $\Tr_{3,8}(\gamma_0)=c_0$ and
$\Tr_{3,19}(\gamma_1)=c_1$.
We also denote by $u_i$ the image of the class defined in
Section~\ref{sec:relations} under the natural map
$(J_{\nu_i}(1))_{\GL_1}\longrightarrow(K_{\nu_i}(1))_{\GL_1}$.
The total Singer transfer is multiplicative~\cite{Singer1989}:
a product of source classes in bidegrees $(k,p)$ and $(l,q)$
is sent to the product in
$\Ext_{\A}^{k+l,p+q+k+l}(\F,\F)$.
Its rank-one values therefore give
\begin{equation}\label{eq:exceptional-products}
\begin{aligned}
 \Tr_{4,9}(u_1\gamma_0)
 &=\Tr_{1,1}(u_1)\Tr_{3,8}(\gamma_0)=h_1c_0,\\
 \Tr_{4,22}(u_2\gamma_1)
 &=\Tr_{1,3}(u_2)\Tr_{3,19}(\gamma_1)=h_2c_1.
\end{aligned}
\end{equation}
Both products on the right are nonzero by Lin
\cite[Theorem~1.3(1)]{Lin2008}.

The invariant dimensions in Lemma~\ref{lem:initial-invariants},
the comparison of $J_n(4)$ with $K_n(4)$ in
Proposition~\ref{prop:KJ}, and the transfer calculation of
Proposition~\ref{prop:adams-injection} determine the domain
dimensions and the transfer on spike coinvariants.
Lemma~\ref{lem:nonadams-targets} identifies the quotient of the
target by its Adams subalgebra component.
The detections above and H\`a's detection of $e_0$
\cite[Theorem~1.1(ii)]{Ha2007} supply the remaining image classes.
With these results, we are now in a position to complete the proof
of Theorem~\ref{thm:main}, our main result formulated in the Introduction.

\begin{proof}[Proof of Theorem~\ref{thm:main}]
Suppose first that $n\notin\{9,17,21,22,46\}$.
Proposition~\ref{prop:KJ} gives $K_n(4)=J_n(4)$, and
Proposition~\ref{prop:adams-injection} therefore identifies
$\Tr_{4,n}$ with an isomorphism from its domain onto
$\mathscr H^{4,n+4}$.
The reductions in the proof of Proposition~\ref{prop:adams-injection}
give the image bases in Tables~\ref{tab:II} and~\ref{tab:III}.
Lemma~\ref{lem:nonadams-targets} then shows that the cokernel
is zero unless $n=256$.
At that degree,
\[
 \im\Tr_{4,256}=\Span\{h_1^2h_7^2\},
 \qquad
 \Ext_{\A}^{4,260}(\F,\F)
 =\Span\{h_1^2h_7^2,D_3(2)\}.
\]
Consequently,
\[
 \coker\Tr_{4,256}
 =
 \frac{\Span\{h_1^2h_7^2,D_3(2)\}}
      {\Span\{h_1^2h_7^2\}}
 =
 \Span\{\overline{D_3(2)}\},
\]
where the bar denotes the class in this quotient.

It remains to consider the five exceptional degrees.
At $n=21,46$, Lemma~\ref{lem:initial-invariants} gives a zero domain,
and the target classification gives a zero target.
At $n=9,17,22$, the same lemma gives a one-dimensional domain,
while the respective targets have bases $h_1c_0$, $e_0$, and
$h_2c_1$.
Equation~\eqref{eq:exceptional-products} proves that
$\Tr_{4,9}$ and $\Tr_{4,22}$ are nonzero.
For $n=17$, H\`a proved in
\cite[Theorem~1.1(ii)]{Ha2007} that
$e_0\in\im\Tr_{4,17}$; so this map is nonzero as well.
Each is therefore a nonzero linear map between one-dimensional
$\F$-vector spaces and hence an isomorphism.
This completes all cases in Tables~\ref{tab:II} and~\ref{tab:III}.
\end{proof}

\appendix
\section{Exact identities for the parity operations}\label{app:identities}
The comparison theorem uses equalities of subspaces in a fixed
coordinate space. We specify the corresponding matrices and the
equalities here so that the finite part of the argument can be
verified separately from its induction in
Section~\ref{sec:comparisons}. All coefficients and ranks are
over $\F$.

Fix the following representatives of $B\backslash G$.
For a complete flag of row-vector subspaces
\[
 0=\mathcal F_0\subset\mathcal F_1\subset\cdots
 \subset\mathcal F_4=\F^4,\qquad \dim\mathcal F_i=i,
\]
encode a vector $(v_1,v_2,v_3,v_4)$ by the integer
$v_1+2v_2+4v_3+8v_4$, representing each entry of $\F$ by $0$ or $1$.
Choose the $i$th row to be the least encoded vector in
$\mathcal F_i\setminus\mathcal F_{i-1}$.
The resulting rows form a basis adapted
to the flag. Two adapted bases of the same flag differ by left
multiplication by a lower triangular matrix, so this construction
chooses exactly one representative of each coset $Bg$.
Order the $315$ representatives lexicographically by their four
encoded rows. The intermediate subspace identities below refer
to this fixed choice.

Use the basis $\{x^b:b\in\mathbb Z_{\geq0}^4,\ |b|\leq3\}$
of $V$ in increasing total degree and then lexicographic order.
For $1\leq j\leq m$, a chosen representative $g$ of $Bg$,
and $b\in\mathbb Z_{\geq0}^4$ with $|b|\leq3$, the coordinate
$(j,Bg,b)$ of $p\in\mathcal V_m$ is the coefficient of $x^b$
in $p_{j,Bg}$.
Order these coordinates first by $j$, then by the chosen coset
order, and then by the stated order of $b$.
Using column vectors, the matrix $M_{\mathbf a,\epsilon}$ of
$T_{\mathbf a,\epsilon}$ has the following entry in row $(j,Bg,c)$
and column $(k,Bh,b)$:
\begin{equation}\label{eq:transition-matrix}
 \begin{split}
 &M_{\mathbf a,\epsilon}
       [(j,Bg,c),(k,Bh,b)]\\
 &\qquad=\delta_{jk}\delta_{Bg,Bh}
 [x^{2c+\delta_{a_j}}]\left(
 x^b\prod_{i:\epsilon_i=1}\left(\sum_{r=1}^4g_{ri}x_r\right)
 \right),\qquad |b|,|c|\leq3.
 \end{split}
\end{equation}
Here $\delta_{\alpha,\beta}$ is $1$ when $\alpha=\beta$ and $0$
otherwise; $\delta_{a_j}$ is the four-coordinate parity vector
defined in Section~\ref{sec:frobenius}.
Thus every entry is the coefficient of a polynomial of degree
at most seven.
If the columns of $N$ span a subspace $U$, form the matrix
obtained by concatenating the sixteen blocks
$M_{\mathbf a,\epsilon}N$ horizontally, with
$\epsilon\in\{0,1\}^4$ in lexicographic order.
Its column space is $L_{\mathbf a}(U)$.
The map $\pi$ selects the coordinates with $c=(0,0,0,0)$.
The initial column spanning $U_m$ has entry $1$ in these constant
coordinates and entry $0$ in every other coordinate.
These formulas determine every matrix used below.

The accompanying calculation report in~\cite{Data} specifies
the stored bases and describes their conversion to the coordinate
order above.

\begin{lemma}\label{lem:subspace-identities}
The operations in~\eqref{eq:subspace-operation} satisfy
\begin{equation}\label{eq:II-subspace-identities}
\begin{aligned}
 L_{(3,3)}^4U_2&=L_{(3,3)}^5U_2,\\
 L_{(1,3)}L_{(3,3)}^3U_2
   &=L_{(1,3)}L_{(3,3)}^4U_2,\\
 L_{(1,2)}^3L_{(1,3)}L_{(3,3)}^aU_2
   &=L_{(1,2)}^4L_{(1,3)}L_{(3,3)}^aU_2
     &&(a=1,2,3).
\end{aligned}
\end{equation}
For the second family,
\begin{equation}\label{eq:III-subspace-identities}
\begin{aligned}
 L_{(2,3,3)}^4L_{(2,4,4)}U_3
   &=L_{(2,3,3)}^5L_{(2,4,4)}U_3,\\
 L_{(2,1,3)}L_{(2,3,3)}^3L_{(2,4,4)}U_3
   &=L_{(2,1,3)}L_{(2,3,3)}^4L_{(2,4,4)}U_3,\\
 L_{(1,1,2)}^3L_{(2,1,3)}L_{(2,3,3)}^aL_{(2,4,4)}U_3
   &=L_{(1,1,2)}^4L_{(2,1,3)}L_{(2,3,3)}^aL_{(2,4,4)}U_3
     &&(0\leq a\leq3).
\end{aligned}
\end{equation}
The remaining two identities are
\begin{equation}\label{eq:III-one-subspace-identities}
\begin{aligned}
 L_{(1,2,3)}^4L_{(1,2,4)}L_{(2,4,4)}U_3
   &=L_{(1,2,3)}^5L_{(1,2,4)}L_{(2,4,4)}U_3,\\
 \pi L_{(1,0,0)}L_{(1,2,1)}L_{(1,2,3)}^3L_{(1,2,4)}L_{(2,4,4)}U_3
   &=\pi L_{(1,0,0)}L_{(1,2,1)}L_{(1,2,3)}^4L_{(1,2,4)}L_{(2,4,4)}U_3.
\end{aligned}
\end{equation}
\end{lemma}
\begin{proof}
For each pair in the statement, construct matrices $A_0,A_1$
whose columns span the two subspaces, using
\eqref{eq:transition-matrix} in the common coordinate order.
Write $[A_0\ A_1]$ for their horizontal concatenation.
Exact elimination gives
\[
 \rank A_0=\rank A_1=\rank[A_0\ A_1]
\]
with the values in Table~\ref{tab:checks}.
The concatenated matrix spans the sum of the two subspaces.
These equalities of ranks prove equality of the subspaces:
each is contained in their sum and has the full dimension of
that sum.
Both basis matrices for each equality are provided in~\cite{Data}.
An independent C++17 implementation reconstructs the same operators
directly from the polynomial formula and verifies all thirteen
equalities; its results are included in the same archive.
\end{proof}

\begin{table}[htbp]
\centering
\caption{Common ranks of the two spanning matrices and their horizontal concatenation.
The last equality is evaluated in $W_3$; every other equality
is evaluated before taking constant terms.}
\label{tab:checks}
\begin{tabular}{@{}llr@{}}
\toprule
Identity&Ambient space&Common rank\\\midrule
\eqref{eq:II-subspace-identities}, first&$\mathcal V_2$&15\\
\eqref{eq:II-subspace-identities}, second&$\mathcal V_2$&87\\
\eqref{eq:II-subspace-identities}, third, $a=1$&$\mathcal V_2$&399\\
\eqref{eq:II-subspace-identities}, third, $a=2$&$\mathcal V_2$&539\\
\eqref{eq:II-subspace-identities}, third, $a=3$&$\mathcal V_2$&560\\\addlinespace
\eqref{eq:III-subspace-identities}, first&$\mathcal V_3$&85\\
\eqref{eq:III-subspace-identities}, second&$\mathcal V_3$&157\\
\eqref{eq:III-subspace-identities}, third, $a=0$&$\mathcal V_3$&582\\
\eqref{eq:III-subspace-identities}, third, $a=1$&$\mathcal V_3$&767\\
\eqref{eq:III-subspace-identities}, third, $a=2$&$\mathcal V_3$&907\\
\eqref{eq:III-subspace-identities}, third, $a=3$&$\mathcal V_3$&928\\\addlinespace
\eqref{eq:III-one-subspace-identities}, first&$\mathcal V_3$&339\\
\eqref{eq:III-one-subspace-identities}, second&$W_3$&125\\
\bottomrule
\end{tabular}
\end{table}

Applying the constant-term map to the products in
Lemma~\ref{lem:family-recursions} gives the dimensions of
$R_{n_0}$, and hence of $J_{n_0}(4)$, used in
Tables~\ref{tab:dimensions-II} and~\ref{tab:dimensions-III}.
This is a separate rank calculation from the thirteen equalities.
The representative set consists of the twelve first-family
pairs $1\leq s\leq3$, $1\leq t\leq4$, the twenty
second-family pairs $2\leq s\leq5$, $1\leq t\leq5$,
and the six pairs $(1,t)$ with $1\leq t\leq6$.
These thirty-eight ranks are obtained from the coefficient maps above.
Their comparison with Sum's cohit dimensions
\cite[Theorem~1.1]{Sum2014} takes place in Proposition~\ref{prop:KJ}.

\subsection*{Data availability}
The computer calculations in Sections~\ref{sec:families}
and~\ref{sec:transfer} and Appendix~\ref{app:identities} were
performed using Python~3, C and C++17, with exact linear algebra
over $\F$. The output data and calculation reports are provided
in the accompanying Zenodo dataset~\cite{Data}.

\end{document}